\documentclass[10pt]{amsart}

\usepackage[english]{babel}
\usepackage[utf8]{inputenc}
\usepackage[T1]{fontenc}
\usepackage{lmodern}
\usepackage{microtype}
\usepackage{indentfirst}
\usepackage{geometry}
\usepackage{enumitem}
\usepackage{amsmath, amssymb, amsthm}
\usepackage{eucal, mathrsfs, bm, bbm}
\usepackage{stmaryrd, mathtools, braket}
\usepackage[all]{xy}
\usepackage{xspace}
\usepackage[foot]{amsaddr}
\usepackage{hyperref}
\usepackage[capitalise,noabbrev]{cleveref}

\hypersetup{
urlbordercolor=1 1 1,
linkbordercolor=0 0 1
}

\setlist[enumerate]{label=\arabic*.}

\theoremstyle{definition}
\newtheorem{defin}{Definition}[section]
\newtheorem{rem}[defin]{Remark}
\newtheorem{exmp}[defin]{Example}
\theoremstyle{plain}
\newtheorem{theor}[defin]{Theorem}
\newtheorem{prop}[defin]{Proposition}
\newtheorem{lem}[defin]{Lemma}
\newtheorem{corol}[defin]{Corollary}
\crefname{defin}{Definition}{Definitions}
\crefname{rem}{Remark}{Remarks}
\crefname{theor}{Theorem}{Theorems}
\crefname{prop}{Proposition}{Propositions}
\crefname{lem}{Lemma}{Lemmas}
\crefname{corol}{Corollary}{Corollaries}
\crefname{claim}{Claim}{Claims}
\crefname{exmp}{Example}{Examples}
\newcommand{\blank}[1][1]{%
\relax\ifcase#1\relax\or{(\rule[.4ex]{1.5ex}{.4pt})}
\else{\rule[.4ex]{1.5ex}{.4pt}}\fi
}
\newcommand{\abscatm}[1]{\mathbb{#1}}

\newcommand{\catcm}{\abscatm{C}}
\newcommand{\catct}{$\catcm$\xspace}
\newcommand{\catem}{\abscatm{E}}
\newcommand{\catet}{$\catem$\xspace}

\newcommand{\conccatm}[1]{\mathsf{#1}}

\newcommand{\catsetm}{\conccatm{{Set}}}
\newcommand{\catsett}{$\catsetm$\xspace}
\newcommand{\cattopm}{\conccatm{Top}}
\newcommand{\cattopt}{$\cattopm$\xspace}
\newcommand{\indxup}[2]{#1_{\textup{#2}}}
\newcommand{\excomm}[1]{\indxup{#1}{ex}}

\newcommand{\excomcm}{\excomm{\catcm}}
\newcommand{\excomct}{$\excomcm$\xspace}
\newcommand{\posrefm}[1]{\indxup{#1}{po}}

\newcommand{\posrefcm}{\posrefm{\catcm}}

\DeclareMathOperator{\hocatm}{Ho}
\newcommand{\caucomm}[1]{\overline{#1}}
\newcommand{\caucomcm}{\caucomm{\catcm}}
\newcommand{\caucomct}{$\caucomcm$\xspace}
\newcommand{\posrefcslm}[1]{\posrefm{(\catcm/#1)}}
\newcommand{\pjK}{\mathcal{K}}

\newcommand{\GsetKlm}{\catsetm_G}
\newcommand{\GsetKlt}{$\GsetKlm$\xspace}
\newcommand{\GsetEMm}{\catsetm^G}
\newcommand{\GsetEMt}{$\GsetEMm$\xspace}

\newcommand{\xykl}{\ar@{-)}}
\newcommand{\xyklpsrel}[2][]{\xykl@<1ex>[#2]#1 \xykl@<-.3ex>[#2]}

\newcommand{\tokl}{\!\xycenterm[C=1em]{\xykl[r]&}\!}

\newcommand{\fn}[1]{{#1}_1}
\newcommand{\gr}[1]{{#1}_2}
\newcommand{\termm}{\mathbf{1}}
\newcommand{\termt}{$\termm$\xspace}
\newcommand{\proj}{\textup{pr}}
\newcommand{\id}{\textup{id}}
\newcommand{\pbkun}[1]{\langle #1 \rangle}
\newcommand{\prodfnct}{(\blank[2]) \times}
\newcommand{\wprod}[1]{\times^{\textup{w}}_{#1}}
\newcommand{\wspfm}{\mathbf{w}}
\newcommand{\wspfxm}{\wspfm_X}
\newcommand{\wspfxt}{$ \wspfxm $\xspace}
\newcommand{\subm}{\textup{Sub}}

\newcommand{\exfnctm}{\Gamma}
\newcommand{\exfnctt}{$\exfnctm$\xspace}
\newcommand{\prjcm}{\mathcal{P}}
\newcommand{\prjct}{$\prjcm$\xspace}
\newdir{|>}{!/4.5pt/@{|}*:(1,-.2)@^{>}*:(1,+.2)@_{>}*!/-3.5pt/@{ }}
\newdir{ (}{!/ 2.5pt/@{ }*!/:(-.1,-1) 5.7pt/@{(}}
\newdir{ >}{!/2pt/@{ }*@{>}}
\newcommand{\xyar}[4]{\ar@{#1}#2|>{\SelectTips{#3}{}\object@{#4}}} 
\newcommand{\toop}{\leftarrow}
\newcommand{\xys}[1]{\xyar{-}{#1}{cm}{>}}
\newcommand{\xydot}[1]{\xyar{.}{#1}{cm}{>}}
\newcommand{\xyblank}[1]{\xyar{}{#1}{cm}{}}
\newcommand{\xymapsto}[1]{\xyar{-}{|<{\object@{|}}#1}{cm}{>}}
\newcommand{\mono}{\hookrightarrow}
\newcommand{\xymonor}[1]{\xyar{^{(}-}{#1}{cm}{>}@<-1pt>}
\newcommand{\xymonol}[1]{\xyar{_{(}-}{#1}{cm}{>}@<1pt>}
\newcommand{\repi}{\, \mbox{\large$\rightarrowtriangle$}\, }
\newcommand{\xyrepi}[1]{\xyar{-}{#1}{xy}{|>}}
\newcommand{\psrel}{\rightrightarrows}
\newcommand{\xypsrel}[2][]{\xyar{-}{}{cm}{>}@<.8ex>[#2]#1 \xyar{-}{}{cm}{>}@<-.4ex>[#2]}
\newcommand{\ftr}[1][]{\overset{#1}{\longrightarrow}}
\newcommand{\ftrmon}[1][]{\xymatrix@=1.5em{\xymonor{}[r]^-{#1}&}}
\newcommand{\xycenterm}[2][=2em]{\vcenter{\hbox{\xymatrix@#1{#2}}}}
\newcommand{\adj}{\dashv}
\newcommand{\adjrot}[1]{\rotatebox[origin=c]{#1}{$\adj$}}

\newcommand{\allprje}{\textup{P}(\catem)}
\newcommand{\ie}{{\itshape i.e.}\xspace}
\renewcommand{\iff}{if and only if\xspace}
\newcommand{\wrt}{with respect to\xspace}
\newcommand{\mltt}{Martin-L\"of type theory\xspace}
\newcommand{\pcover}{\catpt-cover\xspace}
\newcommand{\pcovers}{\catpt-covers\xspace}
\renewcommand{\pcover}{\prjct-cover\xspace}
\renewcommand{\pcovers}{\prjct-covers\xspace}
\newcommand{\pquot}{\prjct-quotient\xspace}
\newcommand{\pquots}{\prjct-quotients\xspace}
\newcommand{\pker}{\prjct-kernel\xspace}

\makeatletter
\@namedef{subjclassname@2020}{\textup{2020} Mathematics Subject Classification}
\makeatother

\title{On the local cartesian closure of exact completions}
\author{Jacopo Emmenegger}
\email{\href{mailto:op.emmen@gmail.com}{op.emmen@gmail.com}}
\address{%
School of Computer Science, University of Birmingham, Birmingham B15 2TT, UK
}
\thanks{%
This work is licensed under CC-BY-NC-ND 4.0. \url{http://creativecommons.org/licenses/by-nc-nd/4.0/}\xspace
}
\subjclass[2020]{18D15; 18E08; 18A35; 18B15; 18E20}
\keywords{Exact completion, local cartesian closure, weak limits, internal projective objects.}

\begin{document}

\begin{abstract}
This paper presents a necessary and sufficient condition on a category \catct with weak finite limits
for its exact completion \excomct to be (locally) cartesian closed.
A paper by Carboni and Rosolini already claimed such a characterisation
using a different property on \catct,
but we shall show that weak finite limits are not enough for their proof to go through.
We shall also indicate how to strengthen the hypothesis for that proof to work.
It will become clear that,
in the case of ex/lex completions, their characterisation is still valid
and it coincides with the one presented here.
\end{abstract}

\maketitle

This paper presents a condition on a category \catct with weak finite limits
which is equivalent to the (local) cartesian closure of its exact completion \excomct.
In the next section we recall some background notions and results,
besides fixing notation.
\Cref{sec:cc} introduces the concepts needed to formulate the condition on \catct
and proves the characterisation of cartesian closure in \cref{thm:ccex}.
\Cref{sec:lcc} does the same for local cartesian closure proving \cref{thm:lccex}.
In \cref{sec:carochar} we discuss the differences between
our condition and the one given by Carboni and Rosolini in~\cite{CarboniRosolini2000},
and show that the characterisation of cartesian closure
based on the latter requires an additional assumption.
\Cref{sec:concl} contains some concluding remarks.

\section{Preliminaries} \label{sec:preliminaries}

\subsection{Exact categories and projective covers}
A category \catet is \emph{exact}~\cite{Barr1971}
if it has finite limits, regular epis are stable under pullback
and every equivalence relation $ r = (r_1,r_2) \colon R \mono X \times X $ fits in a diagram
\begin{equation} \label{eq:exdiag}
\xycenterm{
  R	\xypsrel[^{r_1}]{r}_{r_2}	&	X	\xyrepi{}[r]^q	&	Q	}
\end{equation}
which is exact,
\ie $ q $ is a coequaliser of $ r_1,r_2 $ and $ r_1,r_2 $ is a kernel pair of $ q $.
When we know an arrow to be a regular epi,
we shall write it with a triangle head as for $ q $ above.
For an introduction to exact categories,
see \cite[Ch.\ 2]{Borceux1994II} and \cite[Ch.\ A1]{Johnstone2002}.

A \emph{quasi limit} of a finite diagram in an exact category \catet
is a cone over that diagram such
that the unique arrow into the limit cone is a regular epi.
Similarly, a diagram as~\eqref{eq:exdiag} is \emph{quasi exact}
if $ q r_1 = q r_2 $ and the unique arrow into the kernel of $ q $ is a regular epi.
In particular, a quasi exact diagram is a coequaliser.

\begin{defin} \label{def:covsq}
Let \catet be an exact category.
A \emph{covering square} in \catet is a quasi pullback
\begin{equation} \label{eq:covsq}
\xycenterm{
X	\xys{}[d]_-{\hat{f}} \xyrepi{}[r]^-p
&	A	\xys{}[d]^-f
\\
Y	\xyrepi{}[r]^-q	&	B	}
\end{equation}
where $ q $ (and so $ p $) is a regular epi.
\end{defin}

An object $ X $ \emph{covers} another object $ A $ if there is a regular epi $ X \repi A $,
and we shall refer to regular epis also as covers.
An arrow
\emph{$ \hat{f} \colon X \to Y $ covers $ f \colon A \to B $ via $p$ and $q$}
if they fit in a covering square as~\eqref{eq:covsq}.
We shall just say that $ \hat{f} \colon X \to Y $ covers $ f \colon A \to B $
when the covers $p$ and $q$ are made clear by the context.

Covering squares enjoy the Beck-Chevalley property for subobjects.

\begin{lem}\label{lem:covsqBC}
Let \catet be an exact category.
For any covering square as in \eqref{eq:covsq},
the canonical natural transformation
$ \exists_p \, \hat{f}^* \ftr[\text{\textbf{.}}] f^* \, \exists_q : \subm(Y) \ftr \subm(A) $
is invertible.
\end{lem}

An object $ X $ in a category \catet is called \emph{(regular) projective}
if, for every regular epi $ g \colon A \repi B $ and arrow $ f \colon X \to B $,
there is a \emph{lift} of $ f $ along $ g $,
\ie an arrow $ f' \colon X \to A $ such that $ g f' = f $.

\begin{defin}
\label{def:projcover}
A \emph{projective cover} of an exact category \catet is a full subcategory
$ \prjcm \colon \catcm \ftrmon \catem $ such that
\begin{enumerate}[label=\alph*)]
\item
for every object $ X $ in \catct, $ \prjcm X $ is projective in \catet, and
\item
every object in \catet is covered by an object in \catct,
\ie for every $ A $ in \catet there are $ X $ in \catct
and a regular epi $ \prjcm X \repi A $.
\end{enumerate}
\catet has \emph{enough projectives} if it has a projective cover.
\end{defin}

When an exact category \catet has a projective cover \prjct,
we shall refer to a regular epi $ \prjcm X \repi A $ as a \emph{\pcover of $ A $}.
Every arrow $ f \colon A \to B $ in \catet can be covered by $ \prjcm \hat{f} $
for some $ \hat{f} \colon X \to Y $ in \catct,
just take a \pcover of a pullback of $ f $ and a \pcover of $ Y $.
Abusing terminology, we shall often refer to $ X $ and $ \hat{f} $
as \pcovers of $ A $ and $ f $, respectively.

\begin{exmp}\label{exmp:gset-pjcov}
Every category \catet monadic over the category \catsett of sets is exact with enough projectives
and the full subcategory on the free algebras
is a projective cover~\cite[Theorem 4.3.5]{Borceux1994II}.
This subcategory is equivalent to the Kleisli category
via the comparison functor $\pjK$.
The counit of the free-forgetful adjunction between \catet and \catsett
provides a choice of $\pjK$-covers of objects,
and postcomposition with the unit yields a choice of $\pjK$-covers of arrows.

Consider in particular the topos \GsetEMt of $G$-sets
for a group $G = (G,\blank[2]\cdot\blank[2],1_G,\blank^{-1})$.
Recall that objects of the Kleisli category \GsetKlt are sets,
and arrows $f \colon A \tokl B$ in \GsetKlt
are functions $f = \pbkun{\fn{f},\gr{f}} \colon A \to B \times G$ in \catsett.
The identity on $A$ is the unit $\eta_A$ of the free-forgetful adjunction
and the composite $gf$ of two arrows
$ f\colon A \tokl B $ and $g \colon B \tokl C $
in \GsetKlt is given by the function $\mu_C(g \times G)f$,
where $\mu_C$ is the free action of $G$ on $C \times G$.
The comparison functor $\pjK \colon \GsetKlm \to \GsetEMm$
maps an arrow $ f \colon A \tokl B $
to the morphism of free actions
$\mu_B (f \times G) \colon (A \times G, \mu_A) \to (B \times G, \mu_B).$
For a morphism $g \colon (A,\alpha) \to (B,\beta)$,
the choice of $\pjK$-covers given by the free-forgetful adjunction
is exhibited by the square
\begin{equation}\label{gset-covsq}
\xycenterm[=2.5em]{
(A \times G, \mu_A)	\xys{}[d]_-{g \times G} \xyrepi{}[r]^-{\alpha}
&	(A,\alpha)	\xys{}[d]^-g
\\
(B \times G, \mu_B)	\xyrepi{}[r]^-{\beta}
&	(B,\beta),
}
\end{equation}
where the horizontal arrows are counit components
and $g \times G = \pjK(\eta_B g)$.
In this case the square \eqref{gset-covsq} is not just covering,
but an actual pullback.
\end{exmp}

Let $ \prjcm \colon \catcm \ftrmon \catem $ be a projective cover.
The poset of subobjects in \catet of an object $ \prjcm X $
is isomorphic to the poset reflection $\posrefm{(\catcm/X)} $ of $ \catcm/X $.
More generally, for a tuple $ X_1, \dots , X_n $ of objects in \catct,
let $ \catcm/(X_1, \dots, X_n) $ denote the category of cones over $ X_1, \dots , X_n $,
\ie the comma category $ \Delta \downarrow (X_1, \dots, X_n) $,
where $ \Delta \colon \catcm \ftrmon \catcm^n $ is the diagonal functor and $ (X_1, \dots, X_n) \colon \mathbbm{1} \to \catcm^n $.
Hence, for every $ n > 0 $ and every tuple $ X_1, \dots , X_n $ of objects in \catct,
\begin{equation}\label{eq:subobj}
\subm_{\catem}(\termm) \cong \posrefcm	\qquad\text{and} \qquad
\subm_{\catem}(\prjcm X_1 \times \dots \times \prjcm X_n) \cong%
\posrefcslm{(X_1, \dots, X_n)}, \text{ for } n \geq 1.
\end{equation}
For this and other properties of projective covers that we shall need,
we refer to~\cite{CarboniVitale1998} and \cite{Vitale1994}.

The following proposition is a useful characterisation
of the full subcategory of projectives among all projective covers of
an exact category.

\begin{prop}[\cite{CarboniVitale1998}]\label{prop:maxpjc}
Let $ \prjcm \colon \catcm \ftrmon \catem $
be a projective cover of an exact category \catet.
The following are equivalent.
\begin{enumerate}
\item	Idempotents split in \catct.
\item	The category $ \catcm $ is closed under retracts in \catet.
\item	The category $ \catcm $ is the full subcategory $\allprje$
	of \catet on the projective objects.
\end{enumerate}
\end{prop}

\subsection{Weak finite limits}
Weak limits are defined as usual limits
but without requiring uniqueness of the induced arrow.
More precisely, an object is weakly terminal if every object has an arrow into it,
and a \emph{weak limit} of a diagram $ D $ is a cone over $ D $
which is weakly terminal among cones over $ D $.
For example, a weak product of $ X $ and $ Y $ is a span $ X \toop W \to Y $ such that,
for any span $ X \toop Z \to Y $, there is a (not necessarily unique) arrow $ Z \to W $
making the two triangles commute.
A weak limit is a limit \iff its projections are jointly monic.
Below we list some examples of weak limits.

\begin{exmp} \label{exmp:wlim} \hspace{1em}
\begin{enumerate}[ref=\ref{exmp:wlim}.\arabic*]
\item \label[exmp]{exmp:wlim:set}
Consider the diagram $ X \toop X \times 2 \times Y \to Y $ in \catsett
where $X$ and $Y$ are inhabited sets, $ 2 $ is a set with two elements
and the two functions are the product projections.
It is clearly a weak product in \catsett with non-jointly monic projections.
\item \label[exmp]{gset-wlim}
The Kleisli category \GsetKlt
from \cref{exmp:gset-pjcov} has weak finite limits.
A weak product of two objects $A$ and $B$ in \GsetKlt is
\[
\xycenterm[C=4em]{
A	&
A \times G \times B \times G	\xykl[l]_-{\proj_{1,2}} \xykl[r]^-{\proj_{3,4}}
&	B	}
\]
where $\proj_{i,j} \coloneqq \langle \proj_i , \proj_j \rangle$.
Since postcomposition in \GsetKlt with $\proj_{1,2}$ and $\proj_{3,4}$
involves a free action on $A$ and $B$, respectively,
if the group $G$ is non-trivial
the projections $\proj_{1,2}$ and $\proj_{3,4}$ are not jointly monic in \GsetKlt.
\item \label[exmp]{exmp:wlim:hty}
Homotopy pullbacks in \cattopt, the category of spaces and continuous functions, become weak pullbacks
when mapped in the category $ \hocatm(\cattopm) $ of topological spaces and homotopy classes of continuous maps.
In particular, the weak kernel pair in $ \hocatm(\cattopm) $
of the homotopy class of the universal cover of the circle
$ f \colon \mathbb{R} \to S^1 $
can be computed as the homotopy pullback
\[
\mathbb{R} \times^{\textup{h}}_{S^1} \mathbb{R} \coloneqq%
\Set{ (x,y,h) \in \mathbb{R} \times \mathbb{R} \times (S^1)^{[0,1]} |%
h(0) = f(x) \text{ and } h(1) = f(y) }
\]
together with the two projections onto $\mathbb{R}$.
Since loops in $S^1$ with different winding numbers are not homotopic,
the two projections cannot be jointly monic in $ \hocatm(\cattopm) $.
\end{enumerate}
\end{exmp}

Example 2 above is an instance of the following fact.

\begin{lem}\label{lem:wlim-pjcov}
Let $ \prjcm \colon \catcm \ftrmon \catem $ be a projective cover
of an exact category \catet.
Let $\mathcal{D} \colon \abscatm{X} \ftr \catcm$ be a finite diagram in \catct,
then a cone $ (p_x \colon W \to \mathcal{D}x)_{x \in \abscatm{X}}$
over $\mathcal{D}$ is a weak limit in \catct \iff
$(\prjcm p_x \colon \prjcm W \to \prjcm \mathcal{D}x)_{x \in \abscatm{X}}$
is a quasi limit of $\prjcm \mathcal{D}$ in \catet,
\ie the unique arrow $p \colon \prjcm W \to \textup{lim}(\prjcm \mathcal{D})$
over $\prjcm \mathcal{D}$ is a regular epi.
In particular, \catct has finite weak limits.
\end{lem}

\begin{proof}
If $(\prjcm p_x)_x$ is a quasi limit,
we can lift along $p$ a universal arrow into the limit
as in the proof of Proposition 4 in~\cite{CarboniVitale1998}.
The converse is straightforward.
\end{proof}

We stated the lemma for finite diagrams,
but the finiteness assumption is actually irrelevant
as soon as the limit exists in \catet.

\begin{rem}\label{rem:wlim-pjcov}
\Cref{lem:wlim-pjcov} amounts to say that a projective cover
maps weak limits into quasi limits,
and reflects quasi limits into weak limits.
A functor $ \catcm \ftr \catem $ from a category with weak finite limits
into an exact category mapping weak limits into quasi limits
is called \emph{left covering}~\cite{CarboniVitale1998,Vitale1994}.
Left covering functors $ \catcm \ftr \catem $ preserve jointly monomorphic families
and thus all finite limits that happen to exist in \catct.
However, a projective cover
(and so a left covering functor) need not preserve weak finite limits.
\end{rem}

Let \catct be a category with weak finite limits.
Weak pullbacks along an arrow $ f \colon V \to X $ in \catct define a functor
$ f^{*\textup{w}} \colon \posrefcslm{X} \to \posrefcslm{V} $.
In particular,
when $ f $ is a projection of a weak product $ Z \toop V \to X $,
we denote the functor defined by weak pullback along $ f $ as
\[
\wprod{X} \colon \posrefcslm{Z} \longrightarrow \posrefcslm{(Z,X)}
\]
and call it the \emph{weak product functor}.
In the case of a projective cover $ \prjcm : \catcm \ftrmon \catem $,
the weak product functor $ \wprod{X} $ is isomorphic
to the product functor $ \prodfnct \prjcm X \colon%
\subm_{\catem}(\prjcm Z) \ftr \subm_{\catem}(\prjcm Z \times \prjcm X) $
via the isomorphisms in \eqref{eq:subobj}.

\subsection{Cartesian closed exact categories}

Let \catet be an exact category.
If it is cartesian closed, then not only it has exponentials
but all simple products, that is, right adjoints $ \Pi_A $
to pullback along product projections. 
The simple product functor $ \Pi_A $ restricts to subobjects,
endowing the internal logic of \catet with universal quantification $ \forall_A \colon \subm(I \times A) \to \subm(I) $.

\begin{rem}\label{rem:univq}
When \catet is cartesian closed and
has a projective cover $ \prjcm \colon \catcm \ftrmon \catem $,
it follows by~\eqref{eq:subobj} that,
for every $Z,X \in \catcm$,
the weak product functor $ \wprod{X} $ has a right adjoint
$ \forall_X^w \colon \posrefcslm{(Z,X)} \to \posrefcslm{Z} $.
As it was already observed in~\cite{CarboniRosolini2000},
the converse is true as well.
Indeed,
suppose that $ \wprod{X} $ is left adjoint for every $X$ and $Z$ in \catct,
then by \eqref{eq:subobj}
there is $\forall_{\prjcm X}$ right adjoint to $\prodfnct \prjcm X$.
Let $I$ and $A$ be objects in \catet and take \pcovers
$ p \colon \prjcm X \repi A $ and $ q \colon \prjcm Z \repi I $,
which result in a covering square exhibiting
$\proj_1 \colon \prjcm Z \times \prjcm X \to \prjcm Z$ as a cover of
$\proj_1 \colon I \times A \to I$.
It follows that the square
of left adjoints below commutes by \cref{lem:covsqBC}.
To obtain a right adjoint $\forall_A$ to $\prodfnct A$,
we can thus apply Theorem 2 in Section 3.7 of~\cite{BarrWells1985}
to the diagram below.
\[
\xycenterm[=3em]{
\subm(I)	\xys{@<1ex>}[d]^{q^*} \xyblank{|-\adj}[d] \xys{}[r]^-{(\_) \times A}
&	\subm(I \times A)	\xyblank{|-\adj}[d] \xys{@<1ex>}[d]^{(q \times p)^*}
\\
\subm(\prjcm Z)
  \xys{@<1ex>}[u]^{\Sigma_q}
  \xys{@<1ex>}[r]^-{(\_) \times \prjcm X}
  \xyblank{|-\perp}[r]
&	\subm(\prjcm Z \times \prjcm X)
  \xys{@<1ex>}[l]^-{\forall_{\prjcm X}}
  \xys{@<1ex>}[u]^{\Sigma_{(q \times p)}}	}
\]
The right adjoints $\forall_A$ satisfy
the Beck-Chevalley condition since their left adjoints do.
\end{rem}

\subsection{The exact completion}

Carboni and Vitale have shown that
a projective cover $ \prjcm \colon \catcm \ftrmon \catem $ exhibits the exact category \catet
as the free exact category over \catct as a category with weak finite limits~\cite{CarboniVitale1998}.
An essential step in their proof consists in embedding \catct into an exact category \excomct,
the \emph{exact completion of \catct}.
The embedding $ \exfnctm \colon \catcm \ftrmon \excomcm $ is equivalent to a projective cover of \excomct,
namely the image of \exfnctt,
and it is universal among left covering functors from \catct into exact categories.
It follows that every exact category with enough projectives is (equivalent to)
the exact completion of any of its projective covers~\cite[Th.\ 16]{CarboniVitale1998}.
We shall state and prove the characterisation of cartesian closure
for a projective cover $\prjcm \colon \catcm \ftrmon \catem$ of \catet exact,
in order to take advantage of the richer structure of \catet.
Here we briefly recall from~\cite{CarboniVitale1998,Vitale1994}
the construction of \excomct from \catct and some properties
which we shall need in the next sections.

A \emph{pseudo equivalence relation} $ x_1,x_2 \colon \bar{X} \psrel X $ in \catct
consists of two arrows $ x_1,x_2 \colon \bar{X} \to X $
together with arrows
$\rho \colon X \to \bar{X}$, $\sigma \colon \bar{X} \to \bar{X}$
and $\tau \colon W \to \bar{X}$,
where
$ \bar{X} \overset{p_2}{\longleftarrow} W \overset{p_1}{\longrightarrow} \bar{X} $
is a weak pullback of
$ \bar{X} \overset{x_2}{\longrightarrow} X \overset{x_1}{\longleftarrow} \bar{X} $,
such that
\[
x_1 \rho = \id_X, \quad x_2 \rho = \id_X, \quad x_1 \sigma = x_2, \quad x_2 \sigma = x_1, \quad
x_1 \tau = x_1 p_1, \quad\text{and}\quad x_2 \tau = x_2 p_2.
\]
We shall denote a pseudo equivalence relation
just by its legs $x_1,x_2$.
Let $ x_1,x_2 \colon \bar{X} \psrel X $ and $ y_1,y_2 \colon \bar{Y} \psrel Y $
be two pseudo equivalence relations in \catct and let $ f \colon X \to Y $.
A \emph{tracking of $ f $ from $ x_1,x_2 $ to $ y_1,y_2 $}
is an arrow $ \bar{f} $ making the diagram
\begin{equation} \label{eq:track}
\xycenterm[=1em]{
  &	\bar{X}	\xys{}[ddl]_(.75){x_1\!} \xys{}[dr]^(.65){\!x_2} \xys{}[rrr]^{\bar{f}}	&&&%
	\bar{Y}	\xys{}[ddl]_(.75){y_1\!} \xys{}[dr]^(.65){\!y_2}	&\\
  &&	X	\xys{|!{[urr];[dr]}\hole}[rrr]_(.7)f	&&&	Y	\\
  X	\xys{}[rrr]_f	&&&	Y	&&}
\end{equation}
commute.
In case $ f $ has a tracking as above,
we say that it is \emph{extensional from $ x_1,x_2 $ to $ y_1,y_2 $}.
Finally, two arrows $f,f' \colon X \to Y$ in \catct are \emph{$ \bar{Y} $-related}
if there is $h \colon X \to \bar{Y}$ making
\[
\xycenterm{
  &	X	\xys{@/_/}[dl]_f \xys{}[d]^h \xys{@/^/}[dr]^{f'}	&\\
  Y	&	\bar{Y}	\xys{}[l]^{y_1} \xys{}[r]_{y_2}	&	Y	}
\]
commute.
Sometimes we find it convenient to write a commutative diagram
involving pseudo equivalence relations in more compact ways,
as in
\[
\xycenterm[C=2.5em]{
\bar{X}	\xypsrel[^-{x_2}]{d}_{x_1} \xys{}[r]^{\bar{f}}
&	\bar{Y}	\xypsrel[^-{y_2}]{d}_{y_1}
\\
X	\xys{}[r]_f	&	Y	}
\qquad \text{and} \qquad
\xycenterm[C=2.5em]{
\bar{U}	\xypsrel[^-{u_2}]{d}_{u_1} \xypsrel[^-{v_1}]{r}_{v_2}
&	V	\xys{}[d]^k
\\
U	\xys{}[r]_h	&	Z}
\]
where the parallel arrows shall always be the legs of a pseudo equivalence relation.
When we say that such diagrams commute,
we mean that they commute componentwise, that is $ f x_i = y_i \bar{f} $ and $ h u_i = k v_i $ for $ i = 1,2 $, respectively.

Objects of \excomct are pseudo equivalence relations in \catct and
arrows of \excomct between $ x_1,x_2 \colon \bar{X} \psrel X $
and $ y_1,y_2 \colon \bar{Y} \psrel Y $ are
equivalence classes $ [f,\bar{f}] $ of extensional arrows
$ f $ from $ x_1,x_2 $ to $ y_1,y_2 $ together with a tracking $ \bar{f} $,
where $ [f,\bar{f}] $ and $ [f',\bar{f'}] $ are identified if 
$f$ and $f'$ are $\bar{Y}$-related.
The embedding $ \exfnctm \colon \catcm \ftrmon \excomcm $ maps an object $ X $
to the free pseudo equivalence relation on $ X $, namely the pair of identities $ \id_X,\id_X $,
and an arrow $ f \colon X \to Y $ to the equivalence class $ [f,f] $
(which consists of $ f $ alone).
As for every left covering functor,
\exfnctt preserves all finite limits that happen to exist in \catct.

For a projective cover $\prjcm \colon \catcm \ftrmon \catem$,
the equivalence between \excomct and \catet follows from the fact that
\prjct generates \catet via coequalisers,
see~\cite[Proposition 15]{CarboniVitale1998}.
In particular,
for every pseudo equivalence relation $x_1,x_2 \colon \bar{X} \psrel X$ in \catct
there is $p \colon \prjcm X \repi A$ making
$\prjcm \bar{X} \psrel \prjcm X \repi A$ quasi exact and,
conversely, for every $A$ in \catet, there are $p$ and $x_1,x_2$ as above.
We shall say that $p$ is a \emph{\pquot of $x_1,x_2$}
and that $x_1,x_2$ is a \emph{\pker of $p$} if they
form a quasi exact diagram.
It is then clear that an arrow in \catct is extensional
\wrt a pair of pseudo equivalence relations
\iff it induces a (necessarily unique) arrow in \catet
between the \pquots of the two relations.
We shall spell out a particular case in \cref{lem:prespro}.
For details we refer to~\cite{Vitale1994} and, in particular,
to the proofs of Theorems 1.5.2 and 1.6.1 therein.

When the projective cover is the embedding $\exfnctm \colon \catcm \ftrmon \excomcm$,
every object $(X,\bar{X})$ of \excomct has a canonical \exfnctt-cover
$[\id_X,\rho] \colon \exfnctm X \repi (X,\bar{X})$.
If we denote by $x_1,x_2 \colon \bar{X}' \psrel X$ its \exfnctt-kernel,
then $\id_X$ gives rise to an isomorphism
$(X,\bar{X}) \cong (X,\bar{X}')$ in \excomct.

\begin{exmp}\label{gset-equiv}
Here we illustrate the above equivalence in the case of
the projective cover $\pjK \colon \GsetKlm \ftr \GsetEMm$
from \cref{exmp:gset-pjcov}.
The functor $\GsetEMm \ftr \excomm{(\GsetKlm)}$
maps an algebra $(A,\alpha)$ in \GsetEMt
to the pseudo equivalence relation
\begin{equation}\label{gset-peq}
\xycenterm[C=4em]{
A \times G	\xyklpsrel[^-{\id_{A \times G}}]{r}_-{\eta_A \alpha}
&	A	}
\end{equation}
in \GsetKlt.
The morphism $\alpha$ is a $\pjK$-quotient of it.
Moreover,
since a square like the one in~\eqref{gset-covsq} from \cref{exmp:gset-pjcov}
is a pullback, the diagram
\begin{equation}\label{gset-exs}
\xycenterm[C=3em]{
(A \times G \times G,\mu_{A \times G})	\xypsrel[^-{\mu_A}]{r}_-{\alpha \times G}
&	(A \times G,\mu_A)	\xyrepi{}[r]^-{\alpha}
&	(A,\alpha)	}
\end{equation}
is exact.
It is then easy to see that
the $\pjK$-cover $\eta_B f \colon A \tokl B$
of a morphism $f \colon (A,\alpha) \to (B,\beta)$
is extensional from $\id_{A \times G},\eta_A \alpha$
to $\id_{B \times G},\eta_B \beta$,
with tracking $\pbkun{f \times G,1_G} \colon A \times G \tokl B \times G$.
This gives the action of $\GsetEMm \ftr \excomm{(\GsetKlm)}$ on arrows.
Conversely, the functor $\excomm{(\GsetKlm)} \ftr \GsetEMm$
maps a pseudo equivalence relation as in~\eqref{gset-peq}
to its $\pjK$-quotient object $(A,\alpha)$.
The action on arrows is given by the universal property
of the coequaliser~\eqref{gset-exs}.
It follows that an arrow $f \colon A \tokl B$ is extensional
from $\id_{A \times G},\eta_A \alpha$ to $\id_{B \times G},\eta_B \beta$
\iff $\beta f \alpha = \beta \mu_B(f \times G)$
\iff $\beta f$ is a morphism $(A,\alpha) \to (B,\beta)$.
\end{exmp}

\begin{rem}\label{rem:caucom}
In light of \cref{prop:maxpjc},
for every projective cover $\prjcm \colon \catcm \ftrmon \catem$,
the full subcategory $\allprje$ of \catet on the projectives
is equivalent to the splitting of idempotents \caucomct of \catct.
It follows that
two categories with weak finite limits have equivalent exact completions
\iff they have equivalent splitting of idempotents.
We shall freely regard
the splitting of idempotents \caucomct of \catct
as the full subcategory of projectives of \excomct.
\end{rem}

\section{Cartesian closure} \label{sec:cc}

In this section we present the characterisation of cartesian closure
for exact completions of categories with weak finite limits.
We begin with some definitions.

\begin{defin}
Let $(g_i \colon Y \to X_i)_{i=1}^n$ be a span  with $n$ legs
in a category \catct with weak finite limits.
A pair of arrows $ y_1,y_2 \colon \bar{Y} \to Y $
is a \emph{weak kernel pair of the span} $(g_i)_i$
if it is a weak limit of the diagram
\[
\xycenterm[C=1ex@R=2.5em]{
Y	\xys{|-{g_1}}[dr]
	\xyblank{|-{\cdots\ }}[drr]
	\xys{|-{g_n}}[drrr]
&&&&
Y	 \xys{|-{g_1}|!{[llll];[dl]}{\hole}}[dlll]
	 \xyblank{|-{\,\cdots}}[dll]
	 \xys{|-{g_n}}[dl]
\\
&	X_1	& \cdots &	X_n.	&}
\]
\end{defin}

It is well-known that, in a category with pullbacks,
a kernel pair of an arrow is an equivalence relation.
One proves similarly that, in a category with weak pullbacks,
a weak kernel pair of a span is a pseudo equivalence relation.
We shall freely regard weak kernel pairs as pseudo equivalence relations.

\begin{defin}
Let \catct be a category with weak finite limits.
Let $ f \colon Y \to X $ be an arrow
and $ x_1,x_2 \colon \bar{X} \psrel X $ a pseudo equivalence relation in \catct.
An \emph{extensional image of $ f $ in $ x_1,x_2 $} consists of three arrows
$ \bar{y} \colon \bar{Y} \to \bar{X} $ and $ y_1,y_2 \colon \bar{Y} \to Y $
which form a weak limit of 
\[
\xycenterm{
Y	\xys{}[d]_f	&	\bar{X}	\xys{}[dl]^{x_1} \xys{}[dr]_{x_2}
&	Y	\xys{}[d]^f
\\
X	&&	X.	}
\]
\end{defin}

In a category with weak finite limits \catct,
the arrows $ y_1,y_2 $,
which are part of an extensional image,
are the legs of a pseudo equivalence relation
$ y_1,y_2 \colon \bar{Y} \psrel Y $.
The arrow $ \bar{y} \colon \bar{Y} \to \bar{X} $
is a tracking for $ f $ from $ y_1,y_2 $ to $ x_1,x_2 $,
and $ [f,\bar{y}] $ is monic in \excomct.
In particular,
when $ f $ also has tracking $ \bar{f} $
from a pseudo equivalence relation $ z_1,z_2 \colon \bar{Z} \psrel Z $ to $ x_1,x_2 $,
the arrow $ [f,\bar{y}] $ is the image factorisation
of $ [f,\bar{f}] $ in \excomct~\cite{CarboniVitale1998}.
Note also that a weak kernel pair of $ f $
is the same thing as an extensional image of $ f $ in $ \id_X,\id_X $.
In this case, $ [f,\bar{y}] $ is the image factorisation
of $ \exfnctm f \colon \exfnctm Y \to \exfnctm X $ and
the assignation $ f \mapsto [f,\bar{y}] \colon (Y,\bar{Y}) \mono \exfnctm X $ describes the action
of the right-to-left part of the isomorphism $ \subm_{\excomcm}(\exfnctm X) \cong \posrefcslm{X} $.

\begin{defin}
Let \catct be a category with weak finite limits.
Let $ Z \overset{p_1}{\longleftarrow} V \overset{p_2}{\longrightarrow} X $ be a weak product
and let $ y_1,y_2 \colon \bar{Y} \psrel Y $ be a pseudo equivalence relation in \catct.
An arrow $ f \colon V \to Y $ \emph{preserves projections \wrt $ y_1,y_2 $}
if it is extensional from a weak kernel pair of $ p_1,p_2 $ into $ y_1,y_2 $.

Note that this definition does not depend
on the particular weak kernel pair of $ p_1,p_2 $.
When the pseudo equivalence relation on $ Y $ is clear from the context
we just say that $ f $ preserves projections.
\end{defin}

\begin{defin} \label{def:extexp}
Let $ X $ be an object and $ y_1,y_2 \colon \bar{Y} \psrel Y $ be a pseudo equivalence relation
in a category \catct with weak finite limits.
An \emph{extensional exponential of $ y_1,y_2 $ and $ X $}
consists of an object $W$, a weak product
$ W \toop V \to X $ and an arrow $e \colon V \to Y$
which preserves projections \wrt $ y_1,y_2 $
and which is weakly terminal with these properties,
that is,
for every object $ W' $, weak product $ W' \toop V' \to X $ and
arrow $ e' \colon V' \to Y $ preserving projections \wrt $ y_1,y_2 $,
there are arrows $ h \colon W' \to W $ and $ k \colon V' \to V $ making the diagram
\[
\xycenterm{
&	W'	\xydot{}[dl]_-h
&&	V'	\xys{}[ll] \xydot{}[dl]_-k \xys{@/^/|-\hole}[ddl]^(.7){e'}%
		\xys{@/^/}[dr]	&\\
W	&&	V	\xys{}[ll] \xys{}[d]_e \xys{}[rr]	&&	X	\\
	&&	Y	&&		}
\]
commute.
The arrow $ e $ is called \emph{extensional evaluation}.

A category with weak finite limits \emph{has extensional exponentials} if,
for every pseudo equivalence relation $ y_1,y_2 $ and for every object $ X $,
there is an extensional exponential of $ y_1,y_2 $ and $ X $.
\end{defin}

The name draws from the type
of extensional functions in dependent type theory,
which is an example of extensional exponential in the category of types
described in~\cite[Section 6]{EmmeneggerPalmgren2017}
or in~\cite[Section 7.1]{MaiettiRosolini2013a}.
An example in \GsetKlt will be given in \ref{exmp:gset-extexp},
for the moment we introduce
a slight strengthening of the notion of extensional exponential
and prove two properties of it.

\begin{defin} \label{def:esp}
Let \catct be a category with weak finite limits.
Let $ y_1, y_2 \colon \bar{Y} \psrel Y $ be a pseudo equivalence relation in \catct and
let $ Z \overset{g_1}{\longleftarrow} Y \overset{g_2}{\longrightarrow} X $ be a span
such that both $ g_1 $ and $ g_2 $ coequalise $ y_1,y_2 $.
An \emph{extensional simple product of $ g_1 $ and $ g_2 $ \wrt $ y_1,y_2 $} consists of a commutative diagram
\begin{equation} \label{eq:espdiag}
\xycenterm{
	W	\xys{}[d]_w	&	V	\xys{}[l] \xys{}[d]_e \xys{}[rd]	&		\\
	Z			&	Y	\xys{}[l]^-{g_1} \xys{}[r]_-{g_2}		&	X	}
\end{equation}
where $ W \toop V \to X $ is a weak product and $ e $ preserves projections \wrt $y_1,y_2$,
such that, for every commutative diagram
\begin{equation} \label{eq:espunivd}
\xycenterm{
	W'	\xys{}[d]_{w'}	&	V'	\xys{}[l] \xys{}[d]_{e'} \xys{}[rd]	&		\\
	Z			&	Y	\xys{}[l]^-{g_1} \xys{}[r]_-{g_2}		&	X	}
\end{equation}
where $ W' \toop V' \to X $ is a weak product and $ e' $ preserves projections,
there are arrows $h \colon W' \to W $ and $k \colon V' \to V $ making
\begin{equation} \label{eq:espuniv}
\xycenterm{
&	W'	\xydot{}[dl]_-h \xys{@/^/  |!{[dl];[dr]}\hole}[dddl]^(.7){\!w'}
&&	V'	\xys{}[ll]
		\xydot{}[dl]_-k
		\xys{@/^/ |!{[dl];[dddr]}\hole}[dddl]^(.7){\!e'}
		\xys{}[dddr]	&\\
W	\xys{}[dd]_w
&&	V	\xys{}[ll] \xys{}[dd]_e \xys{}[rrdd]	&&\\
&&&&\\
Z			&&	Y	\xys{}[ll]^-{g_1} \xys{}[rr]_-{g_2}		&&	X	}
\end{equation}
commute.
The arrow $ e $ in \eqref{eq:espdiag} is called \emph{extensional evaluation}.

A category with weak finite limits \emph{has extensional simple products} if,
for every pseudo equivalence relation $ y_1,y_2 \colon \bar{Y} \psrel Y $ and
for every span $ Z \overset{g_1}{\longleftarrow} Y \overset{g_2}{\longrightarrow} X $
such that $ g_1 y_1 = g_1 y_2 $ and $ g_2 y_1 = g_2 y_2 $,
there is an extensional simple product of $ g_1 $ and $ g_2 $ \wrt $ y_1,y_2 $.
\end{defin}

\begin{lem} \label{lem:esp2eexp}
Let \catct be a category with weak finite limits.
If \catct has extensional simple products,
then it has extensional exponentials.
\end{lem}

\begin{proof}
Let $ X $ be an object and $ y_1,y_2 \colon \bar{Y} \psrel Y $ be a pseudo equivalence relation in \catct.
Take first a weak product $ U $ of $ T $, $ X $ and $ Y $, where $ T $ is weakly terminal,
and then a weak limit $ u_1,u_2 \colon \bar{U} \psrel U $, $ \bar{u} \colon \bar{U} \to \bar{Y} $ of the diagram
\[
\xycenterm[R=3em]{
U	\xys{}[d] \xys{|!{[d];[rr]}{\hole}}[dr]
	\xys{|!{[d];[rr]}{\hole}|!{[dr];[rrrr]}{\hole}}[drrr]
&&	\bar{Y}	\xys{}[dll]
		\xys{|!{[dl];[rr]}\hole |!{[dr];[rr]}\hole}[drr]
&&	U	\xys{}[dlll] \xys{}[dl] \xys{}[d]
\\
Y	&	T	&&	X	&	Y.	}
\]
It easy to check that the pair $ u_1,u_2 $
form a pseudo equivalence relation on $ U $.
In fact a weak limit of the above diagram
can also be constructed taking first
a weak kernel pair of the span $ T \toop U \to X $,
and then an extensional image in $ y_1,y_2 $
(or vice versa, first an extensional image and then a weak kernel pair).
Informally, two elements in $ U $ are $ \bar{U} $-related
if their $ T $ and $ X $ components coincide
and their $ Y $ components are $ \bar{Y} $-related.
The two product projections in $ T $ and $ X $
coequalise $ u_1,u_2 $ by construction
and a straightforward computation shows that
an extensional simple product of $ T \toop U \to X $ \wrt $ u_1,u_2 $
is an extensional exponential of $ y_1,y_2 $ and $ X $.
\end{proof}

\begin{lem} \label{lem:wunivq}
Let \catct be a category with weak finite limits.
If \catct has extensional simple products,
then it has right adjoints to weak product functors.
\end{lem}

\begin{proof}
Consider the weak product functor
$ \wprod{X} \colon \posrefcslm{Z} \to \posrefcslm{(Z,X)}$
mapping $ f \colon Y \to Z $ to the span
$ Z \overset{f p_1}{\longleftarrow} V \overset{p_2}{\longrightarrow} X$,
where $ Y \overset{p_1}{\longleftarrow} V \overset{p_2}{\longrightarrow} X$
is a weak product.
We use extensional simple products to define a functor
$ \forall^{\textup{w}}_X $ going the other way.
Let $ Z \overset{g_1}{\longleftarrow} Y \overset{g_2}{\longrightarrow} X $ be a span
and let $ y_1,y_2 \colon \bar{Y} \psrel Y $ be its weak kernel pair.
Take an extensional simple product of $ g_1,g_2$ \wrt $ y_1,y_2 $ as~\eqref{eq:espdiag}
and define $ \forall^{\textup{w}}_X [g_1,g_2] \coloneqq [w] $.
A simple verification shows that
the universal property of extensional simple products exhibits the extensional evaluation
$ e \colon V \to Y $ as the counit of the adjunction $ \wprod{X} \adj \forall^{\textup{w}}_X $:
one only needs to observe that, in a commutative diagram as \eqref{eq:espunivd}
the arrow $ e' \colon V' \to Y $ always preserves projections
\wrt the weak kernel pair of $ g_1,g_2 $.
\end{proof}

Let us fix for the rest of the section an exact category \catet
together with a projective cover $\prjcm \colon \catcm \ftrmon \catem$
exhibiting \catet as the exact completion of \catct.
We can characterise arrows preserving projections as follows.

\begin{lem}\label{lem:prespro}
Let $ Z \toop V \to X $ be a weak product and
$ y_1,y_2 \colon \bar{Y} \psrel Y $ a pseudo equivalence relation in \catct
and let $ q \colon \prjcm Y \repi B $ be a \pquot of it in \catet.
The following are equivalent for an arrow $ f \colon V \to Y $ in \catct.
\begin{enumerate}
\item
The arrow $ f $ preserves projections \wrt $ y_1,y_2 $.
\item
There is a (necessarily unique) arrow $ g \colon \prjcm Z \times \prjcm X \to B $ in \catet such that the square
\[
\xycenterm{
  \prjcm V	\xys{}[d]_{\prjcm f} \xyrepi{}[r]	&	\prjcm Z \times \prjcm X	\xys{}[d]^g	\\
  \prjcm Y	\xyrepi{}[r]^q		&	B				}
\]
commutes.
\end{enumerate}
\end{lem}

\begin{proof}
Immediate from the definition.
\end{proof}

One half of the characterisation is straightforward.

\begin{prop} \label{prop:cc2esp}
If \catet is cartesian closed,
then \catct has extensional simple products.
\end{prop}

\begin{proof}
Let $ \bar{Y} \psrel Y $ be a pseudo equivalence relation in \catct
and let $ q \colon \prjcm Y \repi B $ be its \pquot in \catet.
Any span $ Z \toop Y \to X $ in \catct whose legs coequalise $ \bar{Y} \psrel Y $
induces an arrow $ g \colon B \to \prjcm Z \times \prjcm X $ in \catet.
Let $ w \colon W \to Z $ be the reflection in \catct
of the composite of the simple product
$ \Pi_{\prjcm X} g \colon \Pi_{\prjcm X} B \to \prjcm Z $
with a \pcover $ p \colon \prjcm W \repi \Pi_{\prjcm X} B $.
The extensional evaluation $e \colon V \to Y$ is obtained covering the composite
$ \textup{ev}(p \times \id) $:
\[
\xycenterm{
\prjcm V	\xys{}[d]_{\prjcm e} \xyrepi{}[r]
&	\prjcm W \times \prjcm X	\xys{}[d]^-{\textup{ev}(p \times \id)}
\\
\prjcm Y	\xyrepi{}[r]^-q
&	B.	}
\]
The top row exhibits $V$ as a weak product of $W$ and $X$ by \cref{lem:wlim-pjcov},
\cref{lem:prespro} ensures that $e$ preserves projections
and commutativity of \eqref{eq:espdiag} is immediate.
The required universal property follows from
that one of $ \Pi_{\prjcm X} f $ using again \cref{lem:wlim-pjcov,lem:prespro}.
\end{proof}

We can now provide an example of an extensional exponential
in a category with weak finite limits.

\begin{exmp}\label{exmp:gset-extexp}
Consider again the topos of $G$-sets from \cref{exmp:gset-pjcov}.
Let $X$ be any set and let  $\id_{B \times G},\eta_B \beta$
the pseudo equivalence relation
in \GsetKlt associated to an action $(B,\beta)$ as in \cref{gset-equiv}.
An exponential in \GsetEMt of $(B,\beta)$ and $\pjK X$ consists
of the set $B^{X \times G}$ and
the action $ \epsilon \colon B^{X \times G} \times G \to B^{X \times G} $
mapping a pair $(f,g_1)$ to the function
$ \epsilon(f,g_1) \colon (x,g_2) \mapsto \beta(f(x,g_2 \cdot g_1^{-1}),g_1) $.
The evaluation
$\textup{ev} \colon B^{X \times G} \times (X \times G) \to B$
in \catsett is a morphism from $\epsilon$ to $\beta$.

An extensional exponential in \GsetKlt of
the pseudo equivalence relation $\id_{B \times G},\eta_B \beta$
and the set $X$
is given by
the set $B^{G \times X}$
with extensional evaluation defined by the function
\[
\xycenterm[R=1ex]{
(B^{X \times G} \times G) \times (X \times G)	\xys{}[r]^-e
&	B \times G
\\
(f,g_1,x,g_2)	\xymapsto{}[r]
&	(\epsilon(f,g_1)(x,g_2),1_G).	}
\]
\end{exmp}

We now proceed to prove the converse to \cref{prop:cc2esp}.
We begin with an immediate consequence of \cref{lem:wunivq,rem:univq}.

\begin{lem} \label{lem:univq}
If \catct has extensional simple products,
then \catet has right adjoints to inverse images along products projections.
\end{lem}

The validity of following lemma was already observed in~\cite{CarboniRosolini2000}.

\begin{lem}[Carboni--Rosolini] \label{lem:reduct}
The category \catet is cartesian closed \iff
$\prjcm X$ is exponentiable in \catet for every object $X$ in \catct .
\end{lem}

\begin{proof}
One direction is trivial,
let us then assume that every object in \catct is exponentiable in \catet.
Let $ A $ and $ B $ be two objects in \catet
and take a \pcover $ p \colon \prjcm X \repi A $ and a pseudo equivalence relation $x_1,x_2 \colon \bar{X} \psrel X$ in \catct
such that $ \prjcm \bar{X} \psrel \prjcm X \repi A $ is quasi exact.
Consider the following equaliser
\[
\xycenterm{
E\	\xymonor{}[r]^-i
&	B^{\prjcm X}	\xypsrel[^-{B^{\prjcm x_1}}]{rr}_-{B^{\prjcm x_2}}
&&	B^{\prjcm \bar{X}}.	}
\]
We shall prove that $ E $ is an exponential of $ A $ and $ B $.
The evaluation arrow $ e \colon E \times A \to B $ is obtained
from the universal property of the coequaliser in the top row below,
using commutativity of the solid arrows.
\[
\xycenterm[=2.5em]{
E \times \prjcm \bar{X}	
		\xymonor{}[]+<0ex,-.8em>;[d]_-{i \times \id}
		\xypsrel[^-{\id \times \prjcm x_1}]{r}_-{\id \times \prjcm x_2}
&	E \times \prjcm X	\xymonor{}[]+<0ex,-.8em>;[d]^-{i \times \id}
				\xyrepi{}[r]^-{\id \times p}
&	E \times A	\xydot{}[dd]^e
\\
B^{\prjcm X} \times \prjcm \bar{X}
	\xypsrel[^-{B^{\prjcm x_1} \times \id}]{d}_-{B^{\prjcm x_2} \times \id}
	\xypsrel{r}
&	B^{\prjcm X} \times \prjcm X	\xys{}[dr]_-{\textup{ev}}
&\\
B^{\prjcm \bar{X}} \times \prjcm \bar{X}	\xys{}[rr]_-{\textup{ev}}	&&	B	}
\]

Given $ C \in \catem $ and $ f \colon C \times A \to B $,
there is a unique arrow $ g' \colon C \to B^{\prjcm X} $ making the obvious triangle commute.
This unique arrow factors through $ i \colon E \mono B^{\prjcm X} $ since the left-hand diagram below commutes
and $ B^{\prjcm \bar{X}} $ is an exponential.
The resulting arrow $ g \colon C \to E $ satisfies $ e(g \times A) = f $
because of the commutativity of the right-hand diagram below.
\[
\xycenterm[C=2.7em]{
C \!\times\! \prjcm \bar{X}
  \xypsrel[^-{\id \times \prjcm x_2}]{d}_-{\id \times \prjcm x_1}
  \xys{}[r]^-{g' \!\times\! \prjcm \bar{X}}
&	B^{\prjcm X} \!\!\times\! \prjcm \bar{X}
  \xypsrel{d}
  \xypsrel[^-{B^{\prjcm x_1} \times \id}]{r}_-{B^{\prjcm x_2} \times \id}
&	B^{\prjcm\bar{X}} \!\!\times\! \prjcm\bar{X}	\xys{}[dd]^-{\textup{ev}}
\\
C \!\times\! \prjcm X
  \xyrepi{}[d]_-{C\!\times\!p} \xys{}[r]^-{g' \!\times\! \prjcm X}
&	B^{\prjcm X} \!\!\times\! \prjcm X	\xys{}[dr]^-{\textup{ev}}
&\\
C \!\times\! A	\xys{}[rr]^-f	&&	B	}
\quad
\xycenterm[C=1em]{
C \!\times\! \prjcm X	\xyrepi{}[d]_-{C\!\times\!p}
  \xys{}[dr]^(.7){g' \!\times\! \prjcm X}
  \xys{}@/^/[drrr]^(.7){g \!\times\! \prjcm X}
  \xyrepi{}[rrrr]^-{C\!\times\!p}
&&&&	C \!\times\! A	\xys{}[dd]^{g \!\times\! A}
\\
C \!\times\! A	\xys{}[d]_f
&	B^{\prjcm X} \!\!\times\! \prjcm X	\xys{}[dl]_-{\textup{ev}}
&&	  E \!\times\! \prjcm X
  \xymonol{}[]+<-2.2em,0ex>;[ll]^-{i \!\times\! \prjcm X}
  \xyrepi{}[dr]_-{E\!\times\!p}
&\\
B	&&&&	E \!\times\! A	\xys{}[llll]_-e	}
\]
Uniqueness of $ g $ follows from uniqueness of $ g' $
and monicity of $ i \colon E \mono B^{X} $.
\end{proof}

Before proving our main theorem we formulate one last lemma
which will turn out to be useful in the next section too.

\begin{lem} \label{lem:eexp2cc}
Suppose that \catet has right adjoints to inverse images along product projections.
If \catct has extensional exponentials, then \catet is cartesian closed.
\end{lem}

\begin{proof}
Thanks to \cref{lem:reduct},
it is enough to construct an exponential of $ B $
and $ \prjcm X $ with $ X \in \catcm $.
Take a \pcover $b \colon \prjcm Y \repi B $ of $B$ and
a \pker $ y_1,y_2 \colon \bar{Y} \psrel Y $ of $b$ and
let $ W $ and $ e \colon V \to Y $ be the object and extensional evaluation of
an extensional exponential of $ y_1,y_2 $ and $ X $.
The arrow $e$ induces an arrow $ w \colon \prjcm W \times \prjcm X \to B $
by \cref{lem:prespro},
and the kernel pair of
$ \pbkun{w,\proj_{\prjcm X}} \colon \prjcm W \times \prjcm X \to B \times \prjcm X $
factors through $ \prjcm W \times \prjcm W \times \Delta_{\prjcm X} $ by construction
via an arrow $ k \colon K \mono \prjcm W \times \prjcm W \times \prjcm X$.
Define $ r \coloneqq \forall_{\prjcm X} k \colon R \mono \prjcm W \times \prjcm W $,
where $ \prodfnct \prjcm X \adj \forall_{\prjcm X} $.
The adjunction relation yields that
\begin{equation} \label{ccex:eqrel}
\text{$ t = \pbkun{t_1,t_2}  \colon T \to \prjcm W \times \prjcm W  $ %
factors through $ r $ \iff $ w (t_1 \times \prjcm X) = w (t_2 \times \prjcm X) $.}
\end{equation}
Using \eqref{ccex:eqrel} one proves easily that
$ r $ is an equivalence relation and that
$ w $ coequalises $ r_1 \times \prjcm X $ and $ r_2 \times \prjcm X $.
It follows that
$ r $ has a quotient $ q \colon \prjcm W \repi E $ and that
there is a (unique) arrow
$ v \colon E \times \prjcm X \to B $
such that $ v (q \times \prjcm X) = w $.

Consider now an arrow $ f \colon C \times \prjcm X \to B $.
Take a \pcover $c \colon \prjcm Z \repi C$,
a \pker $z_1,z_2 \colon \bar{Z} \psrel Z$
and a \pcover $ e' \colon V' \to Y $ of the composite
$ f (c \times \prjcm X) \colon \prjcm Z \times \prjcm X \repi C \times X \to B $.
\Cref{lem:wlim-pjcov} ensures that $V'$ with the obvious projections is a
weak product of $Z$ and $X$ in \catct.
The arrow $ e' $ preserves projections by \cref{lem:prespro},
so the universal property of $ e $ yields arrows
$ h \colon Z \to W $ and $ k \colon V' \to V $
such that the diagram in \cref{def:extexp} commutes.
This easily entails that
$ w (\prjcm h \times \prjcm X) = f (c \times \prjcm X) $.
Since the diagram below commutes,
\[
\xycenterm[=2.5em@C=4em]{
\prjcm \bar{Z} \times \prjcm X	\xys{}[d]_-{\prjcm z_1 \times \prjcm X}
				\xys{}[r]^-{\prjcm z_2 \times \prjcm X}
&	\prjcm Z \times \prjcm X	\xyrepi{}[d]^-{c \times \prjcm X}
					\xys{}[r]^-{\prjcm h \times \prjcm X}
&	\prjcm W \times \prjcm X	\xys{}[dd]^-w
\\
\prjcm Z \times \prjcm X	\xyrepi{}[r]^-{c \times \prjcm X}
				\xys{}[d]_-{\prjcm h \times \prjcm X}
&	C \times \prjcm X	\xys{}[dr]^-f
&\\
\prjcm W \times \prjcm X	\xys{}[rr]^-w
&&	B	}
\]
the arrow
$ \pbkun{\prjcm(h z_1), \prjcm(h z_2)} \colon \prjcm \bar{Z} \to \prjcm W \times \prjcm W $
factors through $ r $ because of~\eqref{ccex:eqrel}.
It follows that $ q h \colon \prjcm Z \to E $
coequalises $ \prjcm z_1,\prjcm z_2 $,
and it thus induces an arrow $ \hat{f} \colon C \to E $.
The equation $ v (\hat{f} \times \prjcm X) = f $ follows immediately
once we precompose the two sides with the (regular) epi $ c \times \prjcm X $.

For uniqueness, let $ g \colon C \to E $
be such that $ v (g \times \prjcm X) = f $
and denote with $ l \colon \prjcm Z \to \prjcm W $
the lift of $ g c \colon \prjcm Z \to E $ along $ q \colon \prjcm W \repi E $.
The equation
\[
w (l \times \prjcm X) = v ((q l) \times \prjcm X) =
	v (g \times \prjcm X)(c \times \prjcm X) = f (c \times \prjcm X)
= w (\prjcm h \times \prjcm X),
\]
entails by~\eqref{ccex:eqrel} that
$ \pbkun{l,\prjcm h} \colon \prjcm Z \to \prjcm W \times \prjcm W $
factors through $ r $.
Therefore $ g = \hat{f} $ as desired.
\end{proof}

\begin{theor} \label{thm:ccex}
The category \catet is cartesian closed \iff \catct has extensional simple products.
\end{theor}

\begin{proof}
It is now only a matter of putting all the pieces together.
The left-to-right direction is \cref{prop:cc2esp}.
For the converse, \cref{lem:univq} and \cref{lem:esp2eexp} provide the hypothesis for \cref{lem:eexp2cc},
which then yields the cartesian closure of \catet.
\end{proof}

\begin{rem} \label{rem:psp}
Inspecting the proof of~\cref{lem:wunivq},
one sees that only a specific kind of extensional simple product is used.
In light of this fact,
if we say that a \emph{logical simple product of %
$ Z \overset{g_1}{\longleftarrow} Y \overset{g_2}{\longrightarrow} X $}
is an extensional simple product of $ g_1,g_2 $
\wrt a weak kernel pair of $ g_1,g_2 $,
then we could weaken the hypothesis in \cref{lem:univq}
to requiring only logical simple products.

It is not difficult to see that the converse is true as well:
if \catet has right adjoints to inverse images along products projections,
a logical simple product of $ g_1,g_2 $ in \catct is obtained taking \pcovers
of $ \forall_{\prjcm X} b $ and
of the counit $ \forall_{\prjcm X} b \times \prjcm X \to b $,
where $ b \colon B \mono \prjcm Z \times \prjcm X $
is the image factorisation of $ \pbkun{\prjcm g_1,\prjcm g_2} $ in \catet.
In particular, we could replace the hypothesis on \catet in \cref{lem:eexp2cc}
requiring instead that \catct has logical simple products.
It follows that existence of extensional simple products is
equivalent to existence of extensional exponentials and of logical simple products.
\end{rem}

\begin{rem} \label{rem:wesp}
There is an apparently weaker notion of extensional simple product
which in fact turns out to be equivalent to the one we considered in \cref{def:esp},
in the sense that
existence of one implies existence of the other one.
Say that a diagram as \eqref{eq:espdiag} in \cref{def:esp}
is a \emph{weakly extensional simple product}
if, for every diagram as \eqref{eq:espunivd},
there are three arrows
$h \colon W' \to W $, $k \colon V' \to V $ and $j \colon V' \to \bar{Y} $
making the two diagrams below commute.
\[
\xycenterm{
&	W'	\xys{}[dl]_-{w'} \xys{}[d]^-h
&	V'	\xys{}[l] \xys{}[d]_-k \xys{}[dr]
&&%
V	\xys{}[d]_-e
&	V'	\xys{}[l]_-k \xys{}[d]_-j \xys{}[dr]^-{e'}
&\\
Z	&	W	\xys{}[l]^-w
&	V	\xys{}[l] \xys{}[r]	&	X
&%
Y	&	\bar{Y}	\xys{}[l]^-{y_1} \xys{}[r]_-{y_2}	&	Y	}
\]
That is, the arrow $ k \colon V' \to V $
makes $ e $ and $ e' $ not equal but only $ \bar{Y} $-related,
as witnessed by $j$.
One can define weakly extensional exponentials similarly.

Of course, an extensional simple product is also a weakly extensional simple product,
where the arrow $ V' \to \bar{Y} $ is the composite of
$ e' \colon V' \to Y $ with reflexivity of $ y_1,y_2 $.
The converse is false in general.
Indeed, consider \cref{exmp:gset-extexp}:
the function
$ e' \colon B^{X \times G} \times G \times X \times G \to B \times G $
defined by $ f,g_1,x,g_2 \mapsto \pbkun{f(x,g_2 \cdot g_1^{-1}), g_1} $
gives rise to a weakly extensional exponential in \GsetKlt of $y_1,y_2$ and $X$,
and it is possible to show that it is not an extensional exponential.
Nevertheless, given a weakly extensional simple product of
$ g_1,g_2 $ \wrt $ y_1,y_2 $
we obtain an extensional simple product of $ g_1,g_2 $ \wrt $ y_1,y_2 $
simply replacing the (weakly) extensional evaluation
$ e \colon V \to Y $ with the arrow $ f \colon U \to Y $ in
\[
\xycenterm{
U	\xys{@/^1.5em/}[rr]^-f \xys{}[d] \xys{}[r]
&	\bar{Y}	\xys{}[d]^-{y_1} \xys{}[r]_-{y_2}
&	Y
\\
V	\xys{}[r]_-e
&	Y	&}
\]
where the square is a weak pullback.
It is then clear that $ f $ is an extensional evaluation
once we know that $ W \toop U \to X $ is a weak product and $ f $ preserves projections.
To show these two facts we may assume without loss of generality that
there is a projective cover $\prjcm \colon \catcm \ftrmon \catem$
of an exact category \catet.
In this setting, the arrow $ f $ fits in a covering square
\[
\xycenterm{
\prjcm U	\xys{}[d]_{\prjcm f} \xyrepi{}[r]
&	\prjcm V	\xys{}[d]^{q \prjcm e}
\\
\prjcm Y	\xyrepi{}[r]_q		&	B,			}
\]
where $q$ is a \pquot of $ y_1,y_2 $.
The top cover exhibits $U$ as a weak product of $ W $ and $ X $
by \cref{lem:wlim-pjcov}.
Since $ e \colon V \to Y $ preserves projections,
\cref{lem:prespro} implies that $ q \prjcm e $ factors through
$ \prjcm V \repi \prjcm W \times \prjcm X $
via an arrow $ \hat{e} \colon \prjcm W \times \prjcm X \to B $.
It follows that the square below commutes
\[
\xycenterm{
  U	\xys{}[d]_f \xyrepi{}[r]	&	\prjcm W \times \prjcm X	\xys{}[d]^{\hat{e}}	\\
  Y	\xyrepi{}[r]_q	&	B			}
\]
and an additional application of \cref{lem:prespro} let us conclude
that $ f $ preserves projections.
\end{rem}

\section{Local cartesian closure}\label{sec:lcc}

Whenever $\prjcm \colon \catcm \ftrmon \catem$
is a projective cover and $X$ is an object in \catct,
the induced functor $ \prjcm_{/X} \colon \catcm/X \ftrmon \catem/\prjcm X$
is a projective cover.
Hence we see that \catet is locally cartesian closed \iff
every slice of \catct has extensional simple products:
one simply applies \cref{thm:ccex} to derive the cartesian closure
of slices of the form $ \catem/\prjcm X $,
and then use descent along a regular epi $ p \colon \prjcm X \repi A $
to obtain cartesian closure for an arbitrary slice $ \catem/A $.
Yet, it would be good to also have a ``global'' characterisation,
\ie as a property of \catct instead of each of its slices,
in the same way as existence of dependent products
in a category \catet with finite limits
is equivalent to existence of exponentials in every slice of \catet.
In this section we accomplish this goal.

Let $ X \toop V \to W $ be a weak pullback of a cospan $ X \to Z \toop W $
and let $ y_1,y_2 \colon \bar{Y} \psrel Y $ be a pseudo equivalence relation.
Say that an arrow $ V \to Y $ \emph{preserves (pullback) projections \wrt $ y_1,y_2 $}
if it has a tracking from a weak kernel pair of the pullback span $ X \toop V \to W $ in $ y_1,y_2 $.
Note that the forgetful functor $\catcm/Z \ftr \catcm $ preserves
and reflects pseudo equivalence relations
and that an arrow preserves product projections in $ \catcm/Z $
\iff it preserves pullback projections in \catct.

\begin{defin} \label{def:edp}
Let \catct be a category with weak finite limits.
Let $ y_1, y_2 \colon \bar{Y} \psrel Y $ be a pseudo equivalence relation
and let $ Y \overset{g}{\longrightarrow} X \overset{f}{\longrightarrow} Z $
be a pair of arrows such that $ g y_1 = g y_2 $.
An \emph{extensional dependent product of $ g $ along $ f $ \wrt $ y_1,y_2 $} consists of a commutative diagram
\begin{equation} \label{eq:edpdiag}
\xycenterm{
Y	\xys{}[dr]_g	&	V	\xys{}[l]_e \xys{}[d] \xys{}[r]
&	W	\xys{}[d]^w
\\
&	X	\xys{}[r]_f	&	Z	}
\end{equation}
where the square is a weak pullback and
$ e $ preserves projections \wrt $ y_1,y_2 $,
and such that for any commutative diagram
\begin{equation} \label{eq:edpuniv}
\xycenterm{
Y	\xys{}[dr]_g	&	V'	\xys{}[l]_{e'} \xys{}[d] \xys{}[r]
&	W'	\xys{}[d]^{w'}
\\
&	X	\xys{}[r]_f	&	Z	}
\end{equation}
where the square is a weak pullback and
$ e' $ preserves projections \wrt $ y_1,y_2 $,
there are arrows $ h \colon W' \to W $ and $ k \colon V' \to V $ making
\[
\xycenterm{
&&	V'	\xys{@/_1em/}[dll]_-{e'}
		\xydot{|-k}[dl]
		\xys{@/^/ |!{[dl];[dr]}\hole}[ddl]
		\xys{}[rr]
&&	W'	\xydot{|-h}[dl] \xys{@/^/}[ddl]^-{w'}
\\
Y	\xys{}[dr]_-g	&	V	\xys{}[l]_-e \xys{}[d] \xys{}[rr]
&&	W	\xys{}[d]_-w
&\\
&	X	\xys{}[rr]_-f	&&	Z	&}
\]
commute.
The arrow $ e \colon V \to Y $ is called \emph{extensional evaluation}.

A category with weak finite limits \emph{has extensional dependent products} if,
for every pseudo equivalence relation $ y_1,y_2 \colon \bar{Y} \psrel Y $
and arrows $ Y \overset{g}{\longrightarrow} X \overset{f}{\longrightarrow} Z $
such that $ g y_1 = g y_2 $,
there is an extensional dependent products of $ g $ along $ f $ \wrt $ y_1,y_2 $.
\end{defin}

\begin{lem} \label{lem:esp2edp}
Let \catct be a category with weak finite limits.
If every slice of \catct has extensional simple products,
then \catct has extensional dependent products.
\end{lem}

\begin{proof}
This is straightforward:
an extensional dependent product of
$ Y \overset{g}{\longrightarrow} X \overset{f}{\longrightarrow} Z $ \wrt $ y_1,y_2 $
is given by an extensional simple product of $ f g, g $ \wrt $ y_1,y_2 $
in $ \catcm/Z $,
one only needs to observe that $y_1,y_2$ is a pseudo equivalence relation
in $\catcm/Z$ on $fg$ because $g y_1 = g y_2$.
\end{proof}

The converse is true as well
---it will follow from \cref{thm:lccex} and \cref{prop:cc2esp}.
But what we need is a bit less.

\begin{lem} \label{lem:edp2eexp}
Let \catct be a category with weak finite limits.
If \catct has extensional dependent products, 
then every slice of \catct has extensional exponentials.
\end{lem}

\begin{proof}
Let $ x,y \in \catcm/Z $ and let $ y_1,y_2 \colon \bar{Y} \psrel Y $ be a pseudo equivalence relation in $ \catcm/Z $
on $ y \colon Y \to Z $.
Take a weak pullback
$ X \overset{f_1}{\longleftarrow} U \overset{f_2}{\longrightarrow} Y $
of $ x $ and $ y $ and let $ u_1, u_2 \colon \bar{U} \psrel U $
be the pseudo equivalence relation defined in the weak limit diagram
\begin{equation}\label{eq:extkp}
\xycenterm[C=3em]{
&	\bar{U}	\xys{}[dl]_{u_1}  \xys{}[dr]^{u_2} \xys{}[d]
&\\
U	\xys{}[d]_{f_2} \xys{|!{[d];[r]}\hole}[dr]_(.7){f_1\!}
&	\bar{Y}	\xys{}[dl] \xys{|!{[d];[r]}\hole}[dr]
&	U	\xys{}[dl]^(.7){\!f_1} \xys{}[d]^{f_2}
\\
Y	&	X	&	Y.	}
\end{equation}
Clearly, $ f_1 $ coequalises $ u_1, u_2 $.
Consider then an extensional dependent product of
$ U \overset{f_1}{\longrightarrow} X \overset{x}{\longrightarrow} Z $ \wrt $ u_1, u_2 $
and let $w \colon W \to Z$ and $ e \colon V \to U $
be its object in $\catcm/Z$ and its extensional evaluation, respectively.
We shall prove that $ w $ and $ f_2 e \colon V \to Y $
form an extensional exponential of $ y_1,y_2 $ and $ x $ in $ \catcm/Z $.

Suppose then that there are $ w' \colon W' \to Z $,
a weak pullback
$ X \overset{p'_1}{\longleftarrow} V' \overset{p'_2}{\longrightarrow} W' $
of $ x $ and $ w' $,
and $ e' \colon V' \to Y $ which preserves projections
\wrt $ y_1,y_2 $ and such that $ y e' = x p'_1 $.
It follows by the latter equation
that there is an arrow $ \hat{e} \colon V' \to U $ such that
$ f_1 \hat{e} = p'_1 $ and $ f_2 \hat{e} = e' $.
Using the tracking of $ e' $ and
the universal property of diagram~\eqref{eq:extkp}
one sees that $ \hat{e} $ preserves projections \wrt $ u_1, u_2 $.
By the universal property of $ w $ and $ e $
there are dotted arrows below which make the diagram
\[
\xycenterm[C=2.5em]{
&&&	V'	\xys{@/_1em/}[dlll]_{e'}
		\xys{}[dll]_{\hat{e}}
		\xydot{}[dl]
		\xys{@/^/ |!{[dl];[dr]}\hole}[ddl]
		\xys{}[rr]
&&	W'	\xydot{}[dl] \xys{@/^/}[ddl]^{w'}
\\
Y	\xys{}[dr]_-y	&	U	\xys{|-{\,f_2\,}}[l] \xys{|-{f_1}}[dr]
&	V	\xys{}[l]^e \xys{}[d] \xys{}[rr]
&&	W	\xys{}[d]_w
&\\
&	Z	&	X	\xys{}[l]^-x \xys{}[rr]_-x	&&	Z	&}
\]
commute.
We can thus conclude that
$ w \colon W \to Z $ and $ f_2 e \colon V \to Y $
enjoy the universal property of extensional exponentials in $ \catcm/Z $ as required.
\end{proof}

\begin{lem} \label{lem:wradj}
Let \catct be a category with weak finite limits.
If \catct has extensional dependent products,
then it has right adjoints to weak pullback functors.
\end{lem}

\begin{proof}
The value of the right adjoint along $f$ at a given arrow $g$
is defined taking an extensional dependent product of $g$ along $f$
\wrt a weak kernel pair of $g$.
The proof goes, mutatis mutandis, as that one of \cref{lem:wunivq}.
\end{proof}

\begin{lem} \label{lem:radj}
Let $\prjcm \colon \catcm \ftrmon\catem $ be a projective cover
of an exact category \catet.
If \catct has extensional dependent products,
then \catet has right adjoints to inverse images along any arrow.
\end{lem}

\begin{proof}
From \cref{lem:wradj} and the isomorphisms in \eqref{eq:subobj} we conclude
that \catet has right adjoints to functors
$ (\prjcm f)^* \colon \subm(\prjcm Y) \to \subm(\prjcm X) $
with $ f \colon X \to Y $ in \catct.
The right adjoint to reindexing along an arrow $g$ in \catet
is obtained from the right adjoint to reindexing along a \pcover of $g$
as in \cref{rem:univq}.
\end{proof}

\begin{theor} \label{thm:lccex}
Let $\prjcm \colon \catcm \ftrmon\catem $ be a projective cover
of an exact category \catet.
Then \catet is locally cartesian closed \iff \catct has extensional dependent products.
\end{theor}

\begin{proof}
The left-to-right direction follows from \cref{prop:cc2esp} and \cref{lem:esp2edp}.
For the converse, let $ Z \in \catcm $.
\Cref{lem:radj} entails that $ \catem/\prjcm Z $
has right adjoints along product projections,
and \cref{lem:edp2eexp} ensures that $ \catcm/Z $ has extensional exponentials.
The slice $ \catem/\prjcm Z $ is then cartesian closed by \cref{lem:eexp2cc}.

For the general case,
let $ I $ be an object in \catet,
$ p \colon \prjcm Z \repi I $ a \pcover of $ I $,
$ z_1,z_2 \colon \bar{Z} \psrel Z $ a \pker of $p$
and $ \bar{p} \coloneqq p \prjcm z_1 = p \prjcm z_2 $.
Given $ a \colon A \to I $ and $ b \colon B \to I $,
there are quasi exact diagrams
$ \prjcm \bar{Z} \times_I A \psrel \prjcm Z \times_I A \repi A $
and $ \prjcm \bar{Z} \times_I B \psrel \prjcm Z \times_I B \repi B $
over $ \prjcm \bar{Z} \psrel \prjcm Z \repi I $.
We can form exponentials $ (p^*b)^{p^*a} \colon E \to \prjcm Z $
in $ \catem/\prjcm Z $ and
$ (\bar{p}^*b)^{\bar{p}^*a} \colon \bar{E} \to \prjcm \bar{Z} $
in $ \catem/\prjcm \bar{Z} $.
Pasting pullbacks together
\[
\xycenterm[C=3.5em]{
\bar{E} \times_{\prjcm \bar{Z}} (\prjcm \bar{Z} \prescript{}{i}{\times}_{\prjcm Z} A)
    \xys{}[d] \xys{}[r]
&	\prjcm \bar{Z} \prescript{}{i}{\times}_{\prjcm Z} A
    \xys{}[d]^{\bar{p}^*a} \xys{}[r]^{\prjcm z_i \times A}
&	\prjcm Z \times_I A	\xys{}[d]^{p^*a} \xyrepi{}[r]
&	A	\xys{}[d]^a
\\
\bar{E}	\xys{}[r]^-{(\bar{p}^*b)^{\bar{p}^*a}}
&	\prjcm \bar{Z}	\xys{}[r]^-{\prjcm z_i}
			\xys{}@/_.8em/[];[rr]+<-1ex>_-{\bar{p}}
&	\prjcm Z	\xyrepi{}[r]^-p
&	I	}
\]
we see that, for $ i = 1,2 $,
$\prjcm \bar{Z} \prescript{}{i}{\times}_{\prjcm Z} A
\cong \prjcm \bar{Z} \times_I A$,
$\bar{E} \times_{\prjcm \bar{Z}} (\prjcm \bar{Z} \prescript{}{i}{\times}_{\prjcm Z} A)
\cong \bar{E} \prescript{}{i}{\times}_{\prjcm Z} (\prjcm Z \times_I A)
\cong \bar{E} \times_I A$
and, similarly, that
$E \times_{\prjcm Z} (\prjcm Z \times_I A)
\cong E \times_I A$.
We can thus apply the universal property of $(p^*b)^{p^*a}$ to obtain
$ e_i \colon (\prjcm z_i)(\bar{p}^*b)^{\bar{p}^*a} \to (p^*b)^{p^*a}$
in $\catem/\prjcm Z$ such that the diagram below commutes,
for $ i = 1,2 $.
\[
\xycenterm[C=4em]{
\bar{E} \times_I A
    \xys{@/_1ex/}[ddr]
    \xys{}[drr]^-{\textup{ev}}
    \xys{}[rr]^-{e_i \times A}
&&	E \times_I A
    \xys{@/_1ex/ |!{[d];[dr]}{\hole}}[ddr]
    \xys{}[drr]^-{\textup{ev}}
&&\\
&&	\prjcm \bar{Z} \times_{\prjcm Z} B
    \xys{@/^1ex/}[dl]_-{\bar{p}^*b}
    \xys{}[rr]^(.45){\prjcm z_i \times B}
&&	\prjcm Z \times_I B	\xys{@/^1ex/}[dl]_-{p^*b}
&\\
&	\prjcm \bar{Z}	\xys{}[rr]^-{\prjcm z_i}	&&	\prjcm Z	&&}
\]
It is possible to show that the image factorisation $r \colon R \mono E \times E$
of $ (e_1 , e_2) \colon \bar{E} \to E \times E $
is an equivalence relation in \catet.
Indeed, in the internal logic of \catet,
the exponential object $E$ is a set of functions with values in
$ \prjcm Z \times_I B $.
A pair functions in $E \times E$ is in $r$
\iff their $B$-components are equal and their $Z$-components
are $\bar{Z}$-related.
Let $ q \colon E \repi Q $ be a quotient of $r$
and note that it is also a coequaliser of $e_1,e_2$.
Let $ Q \to I $ be the universal arrow induced by
$p (p^*a)^{p^*b} \colon E \to \prjcm Z \repi I$.
The evaluation is given by the universal property of the coequaliser
\[
\xycenterm[C=3.5em]{
\bar{E} \times_I A	\xypsrel[^{e_1 \times A}]{r}_{e_2 \times A}
&	E \times_I A	\xyrepi{}[r]^{q \times A}	& Q \times_I A	}
\]
applied to
$ (b^*p) \textup{ev} \colon E \times_I A \to \prjcm Z \times_I B \repi B $.
The verification of the universal property
is straightforward and it is left to the reader.
\end{proof}

\section{Carboni and Rosolini's characterisation} \label{sec:carochar}

In this section we investigate the relation between extensional
and weak simple products, introduced in~\cite{CarboniRosolini2000},
and show that the proof of Theorem 2.5 in~\cite{CarboniRosolini2000}
tacitly uses an additional assumption.
In particular,
when \catct has finite limits,
Carboni and Rosolini's characterisation is still valid
and ours reduces to it.

Let us begin recalling few basic facts about internally projective objects.
An object $ P $ in a category with binary products is
\emph{internally (regular) projective} if,
for every object $ C $, regular epi $ A \repi B $ and arrow $ C \times P \to B $,
there are an object $ T $ and arrows $ T \repi C $ and $ T \times P \to A $
such that the following diagram commutes.
\[
\xycenterm{
T \times P	\xyrepi{}[d] \xys{}[r]	&	A	\xyrepi{}[d]	\\
C \times P	\xys{}[r]		&	B			}
\]

\begin{rem}\label{rem:intprj}
Let $ \prjcm \colon \catcm \ftrmon \catem $ be a projective cover
of an exact category \catet.
If $P$ is internally projective in \catet,
we may take $T$ above to be of the form $\prjcm X$ for $X$ in \catct,
so that the arrow $ T \repi C $ is a \pcover of $C$.
Furthermore, the following are equivalent for an object $ X $ in \catct:
\begin{enumerate}
\item
$ \prjcm X $ is internally projective,
\item
the functor $ \prodfnct \prjcm X \colon \catem \to \catem $
preserves projectives,
\end{enumerate}
and, if $ \prjcm X $ is exponentiable in \catet,
we may add the following to the above list of equivalents,
because of the adjunction relation:
\begin{enumerate}[resume]
\item
the exponential functor $ (\_)^{\prjcm X} $ preserves regular epis,
\item
the simple product functor $ \Pi_{\prjcm X} $ preserves regular epis.
\end{enumerate}
\end{rem}

In the last two conditions above
we may also replace ``regular epis'' with ``\pcovers''.
More precisely, we have the following for $ (\_)^{\prjcm X} $
and a similar statement for $ \Pi_{\prjcm X} $.

\begin{lem} \label{lem:expintproj}
Let \catet be an exact category
with a projective cover $ \prjcm \colon \catcm \ftrmon \catem $.
Let $ X $ be an object in \catct and suppose that
$ \prjcm X $ is exponentiable in \catet.
Then the following are equivalent.
\begin{enumerate}
\item
$ \prjcm X $ is internally projective.
\item
For every $ B $ in \catet and every \pcover $ q \colon \prjcm Y \repi B $,
$ q^{\prjcm X} \colon {\prjcm Y}^{\prjcm X} \to B^{\prjcm X} $ is regular epic.
\item
For every $ B $ in \catet there is a \pcover $ q \colon \prjcm Y \repi B $
such that $ q^{\prjcm X} \colon {\prjcm Y}^{\prjcm X} \to B^{\prjcm X} $
is regular epic.
\end{enumerate}
\end{lem}

\begin{proof}
We only need to prove that (3) implies (1), the other two implications being obvious.
Given $ f \colon A \repi B $,
take a lift $ l \colon \prjcm Y \to A $ of $ q \colon \prjcm Y \repi B $ along $ f $.
It follows that $ f^{\prjcm X} l^{\prjcm X} = q^{\prjcm X} $,
which entails that $ f^{\prjcm X} $ is a regular epi.
\end{proof}

The relation between projective objects and internally projective objects
is clarified by the following equivalences.

\begin{lem} \label{lem:proj&prd}
Let \catet be an exact category with enough projectives.
The following equivalences hold.
\begin{enumerate}
\item
Internal projectives are projective \iff \termt is projective.
\item
Projectives are internally projective \iff
projectives are closed under binary products.
\end{enumerate}
\end{lem}

In particular, there are exact categories with enough projectives
where internal projectives and projectives do not coincide.
In the topos of $G$-sets the terminal object is internally projective
(this is always the case in a category with enough projectives)
but not projective.
On the contrary, projectives in \GsetEMt are internally projective,
since binary products of free algebras are projective.
The following is an example of an exact category with enough projectives
whose projectives are not closed under binary products.

\begin{exmp}\label{exmp:projnoint}
Let $\catsetm^M$ be the topos of actions of a monoid $M$ on sets.
Being monadic over \catsett, considerations similar to the topos of $G$-sets apply.
In particular, we shall use the same notation.
One difference that makes this example slightly more involved is that
the square~\eqref{gset-covsq} in \cref{exmp:gset-pjcov}
is not necessarily a pullback.

If the monoid $M$ contains an idempotent element $i$
and no invertible elements,
then binary products of free algebras are not projective.
To show this, it is enough to check that the canonical $\pjK$-cover
\[
\xycenterm[C=4em@R=1ex]{
(M \times M \times M, \mu_{M \times M})	\xyrepi{}[r]^-{\mu \boxtimes \mu}
&	(M \times M, \mu \boxtimes \mu)	\\
(x_1,x_2,x_3)	\xymapsto{}[r]	&	(x_1 \cdot x_3, x_2 \cdot x_3)	}
\]
does not have a section in $\catsetm^M$,
where $\mu$ is another name for the monoid multiplication.
If a function $s = \pbkun{s_1,s_2,s_3} \colon M \times M \to M \times M \times M$
is a section of $\mu \boxtimes \mu$ in \catsett,
then $s(1,1) = (1,1,1)$, $s(1,i) = (1,i,1)$ and $s(i,1) = (i,1,1)$.
On the other hand, $s$ is a morphism of algebras \iff
$s(x_1 \cdot x, x_2 \cdot x) = (s_1(x_1,x_2),s_2(x_1,x_2),s_3(x_1,x_2)\cdot x)$.
It follows that, if $s$ were section of $\mu \boxtimes \mu$ in $\catsetm^M$,
then it would be $s(i,i) = s(1 \cdot i, 1 \cdot i) = (1,1,i)$
and $s(i,i) = s(1 \cdot i, i \cdot i) = (1,i,i)$.
\end{exmp}

We need a few more definitions before discussing Carboni and Rosolini's characterisation.
Let us say that an arrow $ f \colon V \to Y $ out of a weak product
$ Z \overset{p_1}{\longleftarrow} V \overset{p_2}{\longrightarrow} X $
is \emph{determined by projections}
if it preserves projections \wrt $ \id_Y,\id_Y $,
\ie if, for all parallel arrows $ h,k $ into $ V $,
$ p_1 h = p_1 k $ and $ p_2 h = p_2 k $ imply $ f h = f k $.
Mutatis mutandis, the same definition applies to arrows out of a weak pullback.

\begin{rem} \label{rem:detproj}
Let $\prjcm \colon \catcm \ftrmon \catem$ be a projective cover
of an exact category \catet and
$Z \toop V \to X$ a weak product in \catct.
By \cref{lem:prespro},
an arrow $ f \colon V \to Y $ is determined by projections in \catct
\iff $\prjcm f$ factors in \catet through the \pcover
$ \prjcm V \repi \prjcm Z \times \prjcm X $.
It follows that an arrow is determined by projections \iff
it preserves projections \wrt any pseudo equivalence relation on its codomain.
\end{rem}

Let \catct be a category with weak finite limits.
If \catct has binary products,
then for every weak product $Z \toop V \to X$
there is an idempotent $i \colon V \to V$ such that
$ p_1 i = p_1 $ and $ p_2 i = p_2 $ and
$i$ is determined by projections.
The converse does not hold in general,
but we do have the following.

\begin{lem} \label{lem:addass}
Let \catct be a category with weak finite limits.
The following are equivalent for objects $Z$ and $X$ in \catct.
\begin{enumerate}
\item
The product $\exfnctm Z \times \exfnctm X$ is projective in \excomct.
\item
For every weak product
$ Z \overset{p_1}{\longleftarrow} V \overset{p_2}{\longrightarrow} X $
in \catct there is an idempotent $ i \colon V \to V $
determined by projections and
such that $ p_1 i = p_1 $, $ p_2 i = p_2 $.
\end{enumerate}
\end{lem}

\begin{proof}
Let $Z$ and $X$ be objects in \catct.
Suppose that $ \exfnctm Z \times \exfnctm X $ is projective in \excomct,
and consider a weak product
$ Z \overset{p_1}{\longleftarrow} V \overset{p_2}{\longrightarrow} X $.
The \exfnctt-cover
$ \pbkun{\exfnctm p_1, \exfnctm p_2} \colon
\exfnctm V \repi \exfnctm Z \times \exfnctm X $
has then a section
$ [s,\bar{s}] \colon \exfnctm Z \times \exfnctm X \to \exfnctm V $.
By faithfulness of \exfnctt, the composite $sq$ is an idempotent on $V$
such that $p_1 sq = sq $ and $p_2 sq = p_2$, 
and $sq$ is determined by projections by \cref{rem:detproj}.
Conversely,
let $[\id_V,\rho_V] \colon \exfnctm V \repi \exfnctm Z \times \exfnctm X $
be the canonical cover of $ \exfnctm Z \times \exfnctm X $
and let $v_1,v_2 \colon \bar{V} \psrel V$ be its \exfnctt-kernel,
so that $Z \overset{p_1}{\longleftarrow} V \overset{p_2}{\longrightarrow} X$
is a weak product in \catct,
where $p_1,p_2$ are the arrows in \catct
giving rise to the projections of $\prjcm Z \times \prjcm X$,
and $ \exfnctm Z \times \exfnctm X \cong (v_1,v_2 \colon \bar{V} \psrel V)$
in \excomct.
Given $i \colon V \to V$ as in 2, 
there is $\bar{s} \colon \bar{V} \to V$ such that $ i v_1 = \bar{s} = i v_2$,
since $i$ is determined by projections,
and there is $h \colon V \to \bar{V}$ such that $v_1 h = i$ and $v_2 h = \id_V$,
since $p_1 i = p_1$ and $p_2 i = p_2$.
This amounts to say that $\bar{s}$ is a tracking for $i$
from $\prjcm Z \times \prjcm X$ to $\exfnctm V$
and that $[i,\bar{s}]$ in \excomct is a section of $[\id_V,\rho_V]$.
\end{proof}

\begin{prop} \label{prop:addass}
Let \catct be a category with weak finite limits.
The following are equivalent.
\begin{enumerate}
\item
\caucomct has binary products.
\item
For every weak product
$ Z \overset{p_1}{\longleftarrow} V \overset{p_2}{\longrightarrow} X $
in \catct there is an idempotent $ i \colon V \to V $
determined by projections and
such that $ p_1 i = p_1 $, $ p_2 i = p_2 $.
\end{enumerate}
\end{prop}

\begin{proof}
The equivalence follows from \cref{lem:addass,rem:caucom}.
\end{proof}

\begin{corol} \label{corol:undfn}
Let \catct be a category with weak finite limits
and suppose that \caucomct has binary products.
Then every arrow preserving projections
\wrt a pseudo equivalence relation $y_1,y_2$
is $\bar{Y}$-related to an arrow determined by projections.
\end{corol}

\begin{proof}
Let
$Z \overset{p_1}{\longleftarrow} V \overset{p_2}{\longrightarrow} X$
be a weak product,
$i \colon V \to V$ the idempotent from \cref{prop:addass}.2
and let $f \colon V \to Y$ preserve projections \wrt
a pseudo equivalence relation $y_1,y_2 \colon \bar{Y} \psrel Y$.
The arrow $f i \colon V \to Y$ is determined by projections as $i$ is.
Note that $i$ is $\bar{V}$-related to $\id_Y$,
where $\bar{V} \psrel V$ is a weak kernel pair of $p_1,p_2$,
as in the second part of the proof of \cref{lem:addass}.
It follows that the arrow $fi$ is $\bar{Y}$-related to $f$
since $f$ preserves projections.
\end{proof}

Note that the converse of \cref{corol:undfn} holds as well:
given its conclusion, an idempotent on a weak product $V$ is obtained as the arrow
$\bar{V}$-related to the identity on $V$,
where $\bar{V} \psrel V$ is the weak kernel pair of
the weak product projections.

\begin{defin}[\cite{CarboniRosolini2000}, 2.1]
A \emph{weak simple product} of a span
$ Z \overset{g_1}{\longleftarrow} Y \overset{g_2}{\longrightarrow} X $
in a category \catct with weak finite limits
is a commutative diagram as~\eqref{eq:espdiag}
where $ W \toop V \to X $ is a weak product and the arrow $ e \colon V \to Y $,
called \emph{weak evaluation}, is determined by projections,
such that, for every commutative diagram as~\eqref{eq:espunivd}
where $ W' \toop V' \to X $ is a weak product
and $ e' \colon V' \to Y $ is determined by projections,
there are arrows $ h \colon W' \to W $ and $ k \colon V' \to V $
making diagram \eqref{eq:espuniv} commutative.
\end{defin}

Clearly, weak simple products
are extensional simple products \wrt free pseudo equivalence relations,
\ie those of the form $ \id_Y,\id_Y $.
One can define a \emph{weak exponential} of $ Y $ and $ X $
to be an extensional exponential of $ Y $ and $ X $ \wrt $ \id_Y,\id_Y $.
It consists of an object $ W $ together with a weak product $ W \toop V \to X $
and an arrow $ V \to Y $ determined by projections,
which is weakly terminal among arrows $ V' \to Y $  determined by projections,
where $ W' \toop V' \to X $ is a weak product.
As it may be expected, in a category with weak finite limits,
weak exponentials can be constructed from weak simple products.
The relation with extensional simple products is clarified
by the following proposition together with \cref{rem:wesp},
where the notion of \emph{weakly extensional simple product}
is introduced and discussed.

\begin{prop} \label{prop:wsp2wesp}
Let \catct be a category with weak finite limits.
If \caucomct has binary products, then
\begin{center}
\begin{minipage}{.05\columnwidth}
$ (*) $
\end{minipage}
\begin{minipage}{.9\columnwidth}
for every span $ Z \overset{g_1}{\longleftarrow} Y \overset{g_2}{\longrightarrow} X $,
a weak simple product of $ g_1,g_2 $
is a weakly extensional simple product of $ g_1,g_2 $
\wrt any pseudo equivalence relation $ y_1,y_2 \colon \bar{Y} \psrel Y $
such that $ g_1 y_1 = g_1 y_2 $ and $ g_2 y_1 = g_2 y_2 $.
\end{minipage}
\end{center}
\medskip

Conversely, if \catct has weak simple products,
then $(*)$ implies that \caucomct has binary products.
\end{prop}

\begin{proof}
Suppose that \caucomct has binary products and
consider a weak simple product as in diagram~\eqref{eq:espdiag}.
The weak evaluation $e \colon V \to Y$ is determined by projections,
thus it preserves projections \wrt any pseudo equivalence relation.
Suppose a diagram as~\eqref{eq:espunivd} is given,
where $e' \colon V' \to Y$ preserves projections
\wrt a pseudo equivalence relation $y_1,y_2$ as in $(*)$.
By \cref{corol:undfn} there is $f \colon V' \to Y$ determined by projections,
$\bar{Y}$-related to $e'$ and, in addition,
such that the diagram~\eqref{eq:espunivd} commutes with $f$ in place of $e'$.
It follows that the weak evaluation $e$ is also a weakly extensional one.

Conversely,
let $ Z \overset{p_1}{\longleftarrow} V \overset{p_2}{\longrightarrow} X $
be a weak product span
and let $ v_1,v_2 \colon \bar{V} \psrel V $ be the weak kernel pair of $ p_1,p_2 $.
Take a weak simple product of $ p_1,p_2 $
with object $ w \colon W \to Z $, weak product
$W \overset{q_1}{\longleftarrow} U \overset{q_2}{\longrightarrow} X$,
and weak evaluation $ e \colon U \to V $.
By $(*)$ it is a weakly extensional simple product \wrt $ v_1,v_2 $ and,
since the identity on $ V $ preserves projections \wrt $ v_1,v_2 $,
there are arrows
$h \colon Z \to W$, $ k \colon V \to U $ and $j \colon V \to \bar{Y}$
such that $wh = \id_Z$,  $q_1 k = h p_1 $, $q_2 k = p_2$,
$v_1 j = ek$ and $v_2 j = \id_V$.
It follows that $p_1 e k = w q_1 k = p_1$, $p_2 ek = p_2$
and, since $e$ is determined by projections,
that the composite $ek$ is determined by projections
and it is an idempotent on $V$.
The conclusion follows from \cref{prop:addass}.
\end{proof}

Carboni and Rosolini's proof in~\cite{CarboniRosolini2000}
constructs an exponential in the exact completion
using weak exponentials and weak simple products instead of extensional ones.
Unfortunately two steps of the proof require an additional assumption.
These are Lemma 2.6, where a functor is claimed to be right adjoint,
and the last part of the proof of Theorem 2.5, where
it is claimed that, for a projective cover $\prjcm \colon \catcm \ftrmon \catem$,
an exponential $ B^A $ may be obtained as a quotient
of an equivalence relation on $ (\prjcm Y)^{\prjcm X} $,
where $ \prjcm Y \repi B $ and $ \prjcm X \repi A $ are \pcovers.
By \cref{lem:expintproj}, the latter claim holds \iff projectives are internally projective,
and it turns out that also the first claim is equivalent to projectives being internally projective.
This is not always the case when \catct has weak finite limits
as we saw, for instance, in \cref{exmp:projnoint}.

In order to illustrate the situation for Lemma 2.6 in~\cite{CarboniRosolini2000},
let us consider a projective cover
$\prjcm \colon \catcm \ftrmon \catem$ of an exact category \catet.
If \catct has weak simple products,
it is possible to define functors
\[
\wspfxm \colon \subm_{\catem}(Z \times X) \to \subm_{\catem}(Z)
\]
for any two objects $ Z, X $ in \catct as follows.
Given $ a = (a_1,a_2) \colon A \mono \prjcm Z \times \prjcm X $,
the subobject $ \wspfxm(a) \in \subm_{\catem}(Z) $ is the image factorisation
of the arrow $ w \colon W \to Z $ obtained as a weak simple product
of the span
$ Z \overset{g_1}{\longleftarrow} Y \overset{g_2}{\longrightarrow} X $,
where $ p \colon \prjcm Y \repi A $ is a \pcover and
$g_1$ and $g_2$ are the reflections in \catct of $a_1 p$ and $a_2 p$, respectively.
Note that the definition of $ \wspfxm(a) $ does not depend
(up to isomorphism) on the \pcover of $ A $.
The proof of Lemma 2.6 in~\cite{CarboniRosolini2000} claims that
\wspfxt is right adjoint to $ \prodfnct \prjcm X $.
However this is true \iff $ \prjcm X $ is internally projective in \catet.
Indeed, suppose for the moment that \catet is cartesian closed.
Then for every subobject $ a \colon A \mono \prjcm Z \times \prjcm X $
and \pcover $ \prjcm Y \repi A $,
the following commuting diagram
\begin{equation} \label{eq:intproj}
\xycenterm[C=3em]{
\Pi_{\prjcm X} \prjcm Y	\xyrepi{}[d] \xys{}[r]^-{\Pi_{\prjcm X}(q)}
&	\forall_X A	\xymonor{}[d]^-{\forall_{\prjcm X}(a)}
\\
\wspfxm A	\xymonor{}[r]_-{\wspfxm(a)}	&	\prjcm Z	}
\end{equation}
shows that \wspfxt coincides with $ \forall_{\prjcm X} $
\iff the top arrow is a regular epi.
Since this argument does not depend on the particular \pcover of $ A $,
we conclude that $ \wspfxm \cong \forall_{\prjcm X} $
\iff $ \prjcm X $ is internally projective.
More generally, the same result can be proven
only assuming the existence of weak simple products in \catct,
without using the cartesian closure of \catet.

\begin{prop} \label{prop:intprj}
Let $\prjcm \colon \catcm \ftrmon \catem$
be a projective cover of \catet exact
and suppose that \catct has weak simple products.
The following are equivalent.
\begin{enumerate}
\item
For every $ Z $ and $ X $ in \catct,
there is an adjunction
\[
\xycenterm[C=4em]{
\subm(\prjcm Z)
  \xys{}@<-.9ex>[r]_-{\prodfnct \prjcm X}
  \xyblank{|-{\adjrot{90}}}[r]
  \xys{}@<-1.1ex>[r];[]_-{\wspfxm}
&	\subm(\prjcm Z \times \prjcm X).}
\]
\item
For every $ X $ in \catct,
the object $ \prjcm X $ is internally projective in \catet.
\item
Projectives are internally projectives in \catet.
\end{enumerate}
\end{prop}

\begin{proof}
(1 $\Rightarrow$ 2)
We shall prove that for every weak product
$ Z \overset{p_1}{\longleftarrow} V \overset{p_2}{\longrightarrow} X $
there is an idempotent $i \colon V \to V$ such that
$p_1 i = p_1$, $p_2 i = p_2$
and $i$ is determined by projections.
The statement will follow from \cref{prop:addass},
and the equivalence between conditions 1 and 2 in \cref{rem:intprj}.
Let
$ p = (\prjcm p_1,\prjcm  p_2) \colon \prjcm U \repi \prjcm Z \times \prjcm X $
be a \pcover and consider the following pair of diagrams,
where the left-hand one is a weak simple product of $p_1,p_2$ in \catct
and the right-hand one is an image factorisation in \catet.
\[
\xycenterm[C=3em@R=1.5em]{
&	V	\xys{}[r]^-{v_1} \xys{}[dd]^-e \xys{}[ldd]_-{v_2}
&	W	\xys{}[dd]^-w
&	\prjcm W	\xyrepi{}[dr]^-q \xys{}[dd]_-{\prjcm w}
\\
&&&&	B	\xymonol{}[dl]^-{\wspfxm(\id_{Z \times X})}
\\
X	&	U	\xys{}[l]_(.4){p_2} \xys{}[r]^-{p_1}	&	Z
&	\prjcm Z	&}
\]
Since $ \prjcm Z $ is projective,
we can lift the unit
$ \id_{\prjcm Z} \leq \wspfxm(\id_{\prjcm Z} \times \prjcm X) $
to a section $h \colon Z \to W$ of $w$ in \catct.
It follows that there is $ k \colon U \to V $ in \catct
such that $ v_1 k = h p_1 $ and $ v_2 k = p_2 $.
The composite $ e k \colon U \to U $ is easily seen
to be idempotent, determined by projections
and such that $ p_1 e k = p_1 $ and $ p_2 e k = p_2 $.

(2 $\Rightarrow$ 3)
It follows from the fact that
objects in a fixed projective cover are internally projective
\iff all projectives are internally projective.

(3 $\Rightarrow$ 1)
Observe first that $ \wspfxm(a) \times \prjcm X \leq a $
always holds for any $ a \in \subm_{\catem}(\prjcm Z \times \prjcm X) $,
therefore we only need to show that $ b \leq \wspfxm(b \times \prjcm X) $
for every $ b \in \subm(\prjcm Z) $.
To this aim, let $ p \colon \prjcm U \repi B \times \prjcm X $ be a \pcover
and let $ w \colon W \to Z $ be a weak simple product of
$ Z \overset{g_1}{\longleftarrow} U \overset{g_2}{\longrightarrow} X $,
where $g_1$ and $g_2$ are the reflections in \catct of $b \proj_1 p$ and $\proj_2 p$,
so that $ \wspfxm(b \times \prjcm X) $ is an image factorisation of $ \prjcm w $.
Since $ \prjcm X $ is internally projective,
there are $ y \colon \prjcm Y \repi B $
and $ u \colon \prjcm Y \times \prjcm X \to \prjcm U $
such that $ p u = y \times \prjcm X $.
The composition of any \pcover $ \prjcm V \repi \prjcm Y \times \prjcm X $
with $ u $ is an arrow $ V \to U $ determined by projections by \cref{rem:detproj}.
It follows that the weak universal property of weak simple products
provides an arrow $ Y \to W $ over $ Z $ which,
in turn, induces an arrow $ b \to \wspfxm(b \times \prjcm X) $ as required.
\end{proof}

\Cref{prop:intprj} proves that
one of the additional assumptions listed in the remark below
must be added to the statement of Lemma 2.6 in~\cite{CarboniRosolini2000}
in order for its proof to go through.

\begin{rem} \label{rem:addass}
Let $\prjcm \colon \catcm \ftrmon \catem $ be a projective cover of \catet exact
and recall that \caucomct denotes the splitting of idempotents of \catct
and that $\allprje$ denotes the full subcategory of \catet on the projective objects.
The following are equivalent
by \cref{prop:addass,lem:proj&prd}.
\begin{enumerate}
\item
Projectives in \catet are internally projective.
\item
$\allprje \equiv \caucomcm$ has binary products.
\item
For every weak product
$ Z \overset{p_1}{\longleftarrow} V \overset{p_2}{\longrightarrow} X $ in \catct
there is an idempotent $ i \colon V \to V $
which is determined by projections and such that
$ p_1 i = p_1 $ and $ p_2 i = p_2 $.
\end{enumerate}
\end{rem}

The following is Theorem 2.5 in~\cite{CarboniRosolini2000}
with the required assumption.
Note that it can be seen as a direct consequence of
\cref{prop:wsp2wesp} and \cref{thm:ccex}.

\begin{theor}[\cite{CarboniRosolini2000}, 2.5] \label{thm:caro}
Let $\prjcm \colon \catcm \ftrmon \catem $ be a projective cover of \catet exact.
If one of the equivalent conditions from \cref{rem:addass} holds,
then \catet is cartesian closed \iff \catct has weak simple products.
\end{theor}

\begin{corol} \label{corol:all}
Let $\prjcm \colon \catcm \ftrmon \catem $ be a projective cover of \catet exact
and suppose that one of the conditions from \cref{rem:addass} holds.
Then the following are equivalent.
\begin{enumerate}
\item
\catct has weak simple products.
\item
\catct has extensional simple products.
\item
\catet is cartesian closed.
\end{enumerate}
\end{corol}

The case of local cartesian closure presents, at this point, no surprise.
The conditions in \cref{rem:addass} take the following form.

\begin{rem} \label{rem:addass/}
Let $\prjcm \colon \catcm \ftrmon \catem $ be a projective cover of \catet exact.
The following are equivalent.
\begin{enumerate}
\item
For every $ U $ in \catct,
projectives in $ \catem/\prjcm U $ are internally projective.
\item
$\allprje \equiv \caucomcm$ has pullbacks.
\item
For every weak pullback square in \catct
\[
\xycenterm{
V	\xys{}[d]_{p_1} \xys{}[r]^{p_2}	&	X	\xys{}[d]
\\
Z	\xys{}[r]	&	U	}
\]
there is an idempotent $ i \colon V \to V $ which is determined by projections
and such that $ p_1 i = p_1 $ and $ p_2 i = p_2 $.
\end{enumerate}
\end{rem}

We can define a weak dependent product as an extensional dependent product
\wrt free pseudo equivalence relations,
and a weakly extensional dependent product
similarly to the simple version in \cref{rem:wesp}.
The following is proved as \cref{prop:wsp2wesp}.

\begin{prop} \label{prop:wdp2wedp}
Let \catct be a category with weak finite limits.
If \caucomct has pullbacks, then
\begin{center}
\begin{minipage}{.05\columnwidth}
$ (*) $
\end{minipage}
\begin{minipage}{.9\columnwidth}
for every pair
$ Y \overset{g}{\longrightarrow} X \overset{f}{\longrightarrow} Z $,
a weak dependent product of $ f,g $
is a weakly extensional dependent product of $ f,g $
\wrt any pseudo equivalence relation $ y_1,y_2 \colon \bar{Y} \psrel Y $
such that $ g y_1 = g y_2 $.
\end{minipage}
\end{center}
\medskip

Conversely, if \catct has weak dependent products,
then $(*)$ implies that \caucomct has pullbacks.
\end{prop}

Theorem 3.3 in~\cite{CarboniRosolini2000} is then reformulated as follows,
and it can be derived from \cref{prop:wdp2wedp} and \cref{thm:lccex}.

\begin{theor}[\cite{CarboniRosolini2000}, 3.3] \label{thm:caroloc}
Let $\prjcm \colon \catcm \ftrmon \catem $ be a projective cover of \catet exact.
If one of the conditions from \cref{rem:addass/} holds,
then \catet is locally cartesian closed \iff \catct has weak dependent products.
\end{theor}

\begin{corol}
Let $\prjcm \colon \catcm \ftrmon \catem $ be a projective cover of \catet exact
and suppose that one of the conditions from \cref{rem:addass/} holds.
Then the following are equivalent.
\begin{enumerate}
\item
\catct has weak dependent products.
\item
\catct has extensional dependent products.
\item
\catet is locally cartesian closed.
\end{enumerate}
\end{corol}

The equivalent conditions in \cref{rem:addass,rem:addass/} are satisfied, in particular,
whenever \catct has finite limits.
It follows that an ex/lex completion of \catct is (locally) cartesian closed \iff
\catct has weak simple (resp.\ dependent) products.

\section{Conclusions}\label{sec:concl}

It appears from the above discussion
that the ultimate reason for the failure of
Carboni and Rosolini's characterisation
in the case of a category \catct with weak finite limits
is related to the absence of certain idempotents on weak products or,
equivalently, to the fact that there are arrows preserving projections
not related to an arrow determined by projections.
It is then clear why their characterisation holds when \catct has finite limits,
as in this case the required idempotent on a weak product not only exists
but it even splits.
Carboni and Rosolini's characterisation then still improves
on the non-elementary one presented in~\cite{Rosicky1999}
which requires \catct to be infinitary lextensive,
and it does fulfil their main motivation,
that is, providing a common general reason for the local cartesian closure of
the effective topos and the category of equilogical spaces.

More generally,
as the exact completion \excomct is determined
by the splitting of idempotents \caucomct
rather than by \catct,
it is no surprise that Carboni and Rosolini's characterisation is still valid
when it is \caucomct, rather than \catct, to have limits.
This is the case for the topos of $G$-sets,
where the Kleisli category lacks binary products,
but nevertheless exponentials in \GsetEMt can be determined
by weak simple products in \GsetKlt.
We mentioned an instance where this is not possible in \cref{exmp:projnoint},
but we originally noticed the gap in Carboni and Rosolini's proof
when trying to apply the results in~\cite{CarboniRosolini2000}
to a category of types arising in \mltt,
see~\cite{EmmeneggerPalmgren2017}.
It is provable in \mltt that
this category has finite limits \iff
the so-called principle of Uniqueness of Identity Proofs
holds for all of its objects.
But this principle is independent of the theory,
and it is not known whether the equivalent conditions in
\cref{rem:addass/} are equivalent to it or strictly weaker.
In particular, it is not known whether
the splitting of idempotents
of this category of types has pullbacks.

Together with Erik Palmgren, we presented in~\cite{EmmeneggerPalmgren2017}
a different condition on \catct to derive the local cartesian closure of \excomct,
inspired by P.\ Aczel's Fullness Axiom
from Constructive Set Theory~\cite{Aczel1978},
and we applied it to show when an exact completion produces a model of the
Constructive Elementary Theory of the Category of Sets~\cite{Palmgren2012},
a constructive version of Lawvere's ETCS~\cite{Lawvere1965}.
This ``fullness condition'' is rather natural
from a type theoretic point of view and,
as we argue in~\cite{Emmenegger2018}, from a homotopy theoretic one
as it naturally arises, under mild assumptions, in several homotopy categories
including the homotopy categories of spaces and CW-complexes.
However, it is likely to be stronger than local cartesian closure of \excomct.
\Cref{thm:lccex} can then be used together
with results from~\cite{EmmeneggerPalmgren2017} to obtain
a complete characterisation of models of CETCS
in terms of properties of their choice objects.

\section*{Acknowledgements}
I am grateful to Peter LeFanu Lumsdaine for a fruitful discussion
in an early stage of the research.
I would also like to express my gratitude to Pino Rosolini
for his support and interest in this work after I informed him of my results.
Parts of this research
were presented during 2017
at the 5th CMNAA Workshop in Louvain-la-Neuve,
at the CT Conference in Vancouver
and at the XXVI AILA Meeting in Padua.
A complete presentation was given at the first ItaCa Workshop,
which took place in Milan in December 2019.
I gratefully thank the organisers of the four events
for giving me the opportunity to speak.
Finally, I am especially grateful to the anonymous referee
for extremely helpful and illuminating comments on
previous versions of the presentation.


\end{document}